\documentclass[11pt]{amsart}
\usepackage{amssymb}
\usepackage{amsfonts}
\usepackage{amsmath}
\usepackage{graphicx}
\usepackage{xcolor}
\usepackage{mathrsfs}
\usepackage{stmaryrd}
\usepackage{epsfig,color}
\usepackage{blindtext}
\usepackage{enumerate}
\usepackage{hyperref}
\usepackage{url}
\usepackage{bbm}
\usepackage{filecontents}
\usepackage{nicefrac,mathtools}
\usepackage{bm}  
\usepackage[nocompress]{cite}
\DeclareGraphicsExtensions{.pdf,.jpeg,.png}
\usepackage{epstopdf}
\usepackage{cancel} 
\usepackage[normalem]{ulem} 
\usepackage{verbatim} 
\usepackage{enumitem} 
\usepackage{tikz-cd}
\usetikzlibrary{cd}
\usepackage{color}
\usepackage[msc-links, lite]{amsrefs}
\usepackage{geometry}
\usepackage{tikz}
\usetikzlibrary{decorations.markings}
\usetikzlibrary{arrows.meta}

\usepackage{extarrows}
\newtheorem{theorem}{Theorem}[section]
\newtheorem{conjecture}[theorem]{Conjecture}
\newtheorem{proposition}[theorem]{Proposition}
\newtheorem{lemma}[theorem]{Lemma}
\newtheorem{corollary}[theorem]{Corollary}
\newtheorem{question}[theorem]{Question}
\newtheorem{example}[theorem]{Example}

\newtheorem{fact}[theorem]{Fact}

\theoremstyle{definition}
\newtheorem{definition}[theorem]{Definition}

\newtheorem{remark}[theorem]{Remark}

\numberwithin{equation}{section}

\numberwithin{equation}{section}

\makeatletter
\let\save@mathaccent\mathaccent
\newcommand*\if@single[3]{%
\setbox0\hbox{${\mathaccent"0362{#1}}^H$}%
\setbox2\hbox{${\mathaccent"0362{\kern0pt#1}}^H$}%
\ifdim\ht0=\ht2 #3\else #2\fi
}

\makeatother

\makeatletter
\newcommand*{\transpose}{%
{\mathpalette\@transpose{}}%
}
\newcommand*{\@transpose}[2]{%
\raisebox{\depth}{$\m@th#1\intercal$}%
}
\makeatother
\usetikzlibrary{hobby}

 \usetikzlibrary{decorations}
\makeatletter
\def\pgfutil@Repeat#1#2{#2\ifnum#1>0
  \expandafter\pgfutil@firstofone\else\expandafter\pgfutil@gobble\fi
  {\expandafter\pgfutil@Repeat\expandafter{\the\numexpr#1-1\relax}{#2}}}
\tikzset{
  dash between/.code args={#1 and #2}{%
    \tikz@addoption{%
      \pgfgetpath\currentpath
      \pgfprocessround{\currentpath}{\currentpath}%
      \pgf@decorate@parsesoftpath{\currentpath}{\currentpath}%
      \pgfmathsetlengthmacro\firstpart{(#1)*\pgf@decorate@totalpathlength}%
      \pgfmathsetlengthmacro\secondpart{(#2-(#1))*\pgf@decorate@totalpathlength}%
      \pgfmathsetlengthmacro\thirdpart{(1-(#2))*\pgf@decorate@totalpathlength}%
      \edef\thirdpart{{\thirdpart}{0pt}}%
      \edef\firstpart{{\firstpart}{0pt}}%
      \pgfmathsetlengthmacro\secondpartlength{\pgfkeysvalueof{/tikz/dash between on}
                                            +(\pgfkeysvalueof{/tikz/dash between off})}%
      \pgfmathtruncatemacro\repetitions{\secondpart/\secondpartlength}%
      \pgfmathsetlengthmacro\secondexpand{\secondpart/\repetitions-\secondpartlength}%
      \edef\secondexpand{\the\dimexpr\pgfkeysvalueof{/tikz/dash between off}+\secondexpand\relax}%
      \edef\secondpart{%
        \pgfutil@Repeat{\the\numexpr\repetitions-1\relax}%
          {{\pgfkeysvalueof{/tikz/dash between on}}{\secondexpand}}%
      }%
      \edef\tikz@temp{\firstpart\secondpart\thirdpart}%
      \expandafter\pgfsetdash\expandafter{\tikz@temp}{+0pt}%
    }
  }
}
\makeatother
\tikzset{
  dash between style/.is choice,
  dash between style/dotted/.style        ={dash between on=\pgflinewidth,dash between off=2pt},
  dash between style/densely dotted/.style={dash between on=\pgflinewidth,dash between off=1pt},
  dash between style/loosely dotted/.style={dash between on=\pgflinewidth,dash between off=4pt},
  dash between style/dashed/.style        ={dash between on=3pt,dash between off=2pt},
  dash between style/loosely dashed/.style={dash between on=3pt,dash between off=6pt},
  dash between style/densely dashed/.style={dash between on=3pt,dash between off=2pt},
  dash between style/no/.style={dash between on=0pt, dash between off=1pt},
  dash between on/.initial=\pgflinewidth,
  dash between off/.initial=2pt,
  middle dotted line/.style={
    thick,
    dash between=.35 and .65}}

\newcommand\QQ{\mathbb{Q}}
\newcommand\CC{\mathbb{C}}
\newcommand\PP{\mathbb{P}}
\newcommand\RR{\mathbb{R}}

\newcommand\ZZ{\mathbb{Z}}

\newcommand\FF{\mathbb{F}}

\DeclareMathOperator{\Aut}{Aut}
\DeclareMathOperator{\homo}{Hom}
\DeclareMathOperator{\Tor}{Tor}

\DeclareMathOperator{\et}{\acute{e}t}

\DeclareMathOperator{\Gal}{Gal}

\def\holim{\qopname\relax m{holim}}

\newcommand{\spec}{\mathrm{Spec}}
\usepackage[all]{xy}

\hypersetup{pdfborder=0 0 0} 

\title{Formal Manifold Structures on Positive Characteristic Varieties}
\author{Runjie Hu, Siqing Zhang}
\newcommand{\Addresses}{{
  \bigskip
  \footnotesize

  Runjie Hu, \textsc{Department of Mathematics, Texas A\&M University,
    College Sta, TX 77843}\par\nopagebreak
  \textit{E-mail address}, \texttt{ runjie.hu@tamu.edu}
  
  Siqing Zhang, \textsc{Department of Mathematics, Yale University, New Haven, CT 06511}\par\nopagebreak
  \textit{E-mail address}, \texttt{siqing.zhang@yale.edu}
}}
\date{}
\begin{document}

\maketitle

\begin{abstract}
In his 1970 ICM report \cite{Sullivan-Galois}, Sullivan proposes the program of $l$-adic formalization of the concept of manifolds. In this program, he claims that smooth positive characteristic varieties should carry $l$-adic formal manifold structures. He also claims the existence of an abelianized Galois symmetry on $l$-adic formal manifold structures. This paper carries out this program, establishes the claims for certain varieties, and relates the abelianized Galois symmetry on $l$-adic formal manifold structures to the Galois symmetry of varieties. Meanwhile, we prove that a simply-connected variety is $l$-adic homotopic equivalent to a simply-connected finite CW complex if and only if the $l$-profinite completion of its \'etale homotopy type admits an $l$-local lifting.
\end{abstract}

\section*{Introduction}

The Galois symmetry of $\overline{\QQ}$ or finite fields is used to prove the Adams conjecture by Sullivan in \cite{Sullivan-adams-conjecture} and Quillen-Friedlander in \cite{Quillen-Adams-Conjecture-1}\cite{Friedlander-Adams-Conjecture}. With this idea and the surgery theory, we $l$-adically complete the concept of manifolds to obtain the so-called $l$-adic formal manifold structures, and construct an abelianized Galois symmetry on them.
Within this framework, we prove that the Frobenius action on positive characteristic varieties, the Galois symmetry on characteristic zero varieties and the abelianized Galois symmetry on $l$-adic formal manifolds can be unified. 

This paper is divided into three parts: firstly we prove that simply-connected varieties are homotopically finite CW complexes in the $l$-adic sense with $l$ away from the characteristic when certain topological conditions hold; then we develop the theory of $l$-adic formal manifolds for simply-connected finite CW complexes and include smooth projective varieties as examples; lastly we unify the Galois symmetry on varieties and that on $l$-adic normal structure sets (intuitively, they are a generalized version of the set of all $l$-adic formal manifolds in an $l$-adic homotopy type).

\textbf{Finiteness results.}

In his letter to Larry Breen, Grothendieck explains his vision on representing the ``homotopy types'' of good schemes by finite polyhedra.\footnote{In \cite{grothendieck2021pursuing}*{p.~45, Appendix \S17}, he writes: ``... In the case of a scheme of finite type on
an algebraically closed field $k$ say, the strongest cohomological and
homotopical finiteness theorem would be expressed precisely in terms
of a fine homotopy type, and would say that the ordinary homotopy
types which are their constituents are essentially “finite polyhedra” - and
even compact manifolds with boundary - or in more precise fashion,
their profinite completions (in the sense of Artin-Mazur) prime to the
characteristic $p$ of $k$ are those of such polyhedra. One sees clearly
how to begin on such programme in characteristic $0$, but one foresees
supplementary amusement, or even mystery, in the case $p>0$, for the
varieties which, even birationally, resist being lifted to characteristic $0$!"} 
In the first part of this paper, we investigate the simply connected case of this vision.

Recall that the $l$-profinite completion of a group is the inverse limit of all finite quotients of $G$ with $l$-primary order. This procedure can be generalized to any space. The fundamental group of the $l$-adic completion $Y^{\wedge}_l$ of a space $Y$ is the $l$-adic completion of $\pi_1(Y)$. If $Y$ is a simply-connected finite CW complex, each homotopy group of $Y^{\wedge}_l$ is the $l$-adic completion of the corresponding homotopy group of $Y$.

The concept of $l$-adic complete spaces can be used to study algebraic varieties. Concretely, for any variety $X$, the \'etale fundamental group and \'etale cohomology can be upgraded to a (pro) homotopy type $X_{\et}$ by Artin-Mazur (\cite{artin-mazur-etale-homotopy}). If $X$ is a complex variety, then the $l$-adic completion $X_{\et,l}^{\wedge}$ of $X_{\et}$ is homotopy equivalent to that of the analytification of $X$.

The first result of this paper is the following.

\begin{theorem}[\textbf{Homotopical Finiteness}, Theorem \ref{finiteness-theorem-1} and Corollaries \ref{finiteness-thereom-on-variety}, \ref{characteristic-0-lifting-homotopically}]\label{thm: one big finite thm}
Let $X$ be a pointed, connected variety over a separably closed field $k$ of characteristic $p\geq 0$.
Let $l$ be a prime number.
Assume that either $l\neq p$ or $X$ is complete.
Further assume that the $l$-adic completion of the \'etale fundamental group $\pi_1^{\et}(X)^{\wedge}_l=0$.
Then
\begin{enumerate}
    \item each homotopy group $\pi_i(X^{\wedge}_{\et,l})$ of the $l$-profinite completion $X^{\wedge}_{\et,l}$ of the \'etale homotopy type $X_{\et}$ is a finitely generated $\widehat{\ZZ}_l$-module;
    \item $X_{\et}$ is $l$-adic homotopy equivalent to a simply-connected finite polyhedron, if and only if it is $l$-adic homotopy equivalent to a simply-connected complex variety, if and only if $X^{\wedge}_{\et,l}$ admits an $l$-local lifting (see Definition \ref{l-local-lifting-definition}).
\end{enumerate}
\end{theorem}

The same results of item (1) also hold for many other cases, such as the \'etale homotopy types of suitable topoi. It can also be deduced from \cite[Theorem 3.4.12]{Lurie-rational-p-profinite}, which is proved using the $\mathbb{E}_{\infty}$-algebra structure on the singular cochains. 
Our proof is more elementary, based on inductive arguments using Postnikov tower and Leray-Serre spectral sequence.  One technical difficulty arises when we try to apply topological results of CW complexes to pro-spaces: homotopy inverse limits do not commute with the homology functor $H_*(-;A)$ even for a finite coefficient $A$ in general (\cite{Goerss-homology-homotopy-inverse-limit}).  
Another difficulty is that Serre's mod-$\mathcal{C}$ Hurewicz theorem does not apply since the class of finitely generated $\widehat{\ZZ}_l$-modules does not satisfy Serre's axioms ($\text{II}_A$) and (III) as in \cite{Serre-mod-C-Hurewicz}. 

In item (2), the notion ``$l$-local lifting''  for a simply-connected, $l$-complete space $Y$ means that it is the $l$-profinite completion of some simply-connected, $l$-local space. An algebraic chain model description for this might be interesting.

The item (2) provides a necessary condition for lifting an $l$-adic-simply-connected positive characteristic variety to a complex variety with finite fundamental group (see Corollary \ref{characteristic-0-lifting-obstruction}).

\textbf{$l$-adic Formal Manifolds}

In the surgery theory, a simply-connected topological manifold is a homotopy type equipped with a Poincar\'e duality and a geometric bundle structure on the canonical spherical fibration induced by the Poincar\'e duality (the so-called Spivak normal spherical fibration). An $l$-adic formal manifold is the $l$-adic completion of the above data for a manifold. More concretely, we first prove the following theorem, which is the $l$-adic version of existence and uniqueness of a homotopy tangent bundle on a Poincar\'e duality space.

\begin{theorem}[Theorems \ref{Z/l-poincare-to-spivak} and \ref{uniqueness-of-l-adic-Spivak}]
Any simply-connected finite CW complex with $\ZZ/l$-coefficient Poincar\'e duality admits a \textbf{mod-$l$ Spivak normal spherical fibration} as in Definition \ref{definition-of-mod-l-spivak}, unique up to the equivalences as in Definition \ref{definition-of-equivalence-of-spivak}.
\end{theorem}
For the uniqueness part, we need the $l$-adic version of Spanier-Whitehead duality theory, which is developed in Section \ref{subsection-l-adic-s-duality}. $l$-adic formal manifolds are defined as follows. This concept is vaguely suggested by Sullivan in \cite{Sullivan-Galois}.

\begin{definition}[Definition \ref{definition-of-l-adic-formal-manifolds}]
An $l$-adic formal manifold structure on a $\ZZ/l$-coefficient Poincar\'e duality space $Z$ is a lifting of its mod-$l$ Spivak normal spherical fibration $\nu_Z$ to a topological bundle of Eulidean spaces in the $l$-adic sense. More explicitly, when $l$ is odd, this is equivalent to a real $K$-theory orientation on $\nu_Z$; when $l=2$, this is equivalent to the vanishing of certain $\ZZ/2$-characteristic classes of $\nu_Z$ and a $\widehat{\ZZ}_2$-coefficient lifting of certain $\ZZ/8$ characteristic classes.
\end{definition}

The following is our third result, which is conjectured by Sullivan in \cite{Sullivan-Galois}.

\begin{theorem}[\textbf{$l$-adic Formal Manifold Structures on Varieties}, Theorem \ref{existence-of-formal-manifold-structure}]\label{formal manifold theorem in intro}
Let $X$ be a connected, smooth, projective variety $X$ of dimension at least $3$ over a separably closed field of characteristic $p\geq 0$. Assume that $X^{\wedge}_{\et,l}$ admits an $l$-local lifting and $\pi_1^{\et}(X)^{\wedge}_l=0$. Then $X$ carries a canonical $l$-adic formal manifold structure for any $l\neq p$.
Furthermore, if $X$ has a characteristic zero lifting $X_{\CC}$,  then this $l$-adic formal manifold structure on $X$ is compatible with the underlying manifold structure of $X_{\CC}$.
\end{theorem}

\textbf{Galois Symmetry}

To study the Galois symmetry, we need the concept of $l$-adic normal structure set. Intuitively, this is the set of all $l$-adic formal manifolds in an $l$-adic homotopy type.

\begin{definition}[Definitions \ref{definition-of-l-adic-homotopy-manifold-structures} and \ref{definition-of-l-adic-normal-structure-set}]
Let $Z$ be a simply-connected finite CW complex with $\ZZ/l$-coefficient Poincar\'e duality, with an $l$-adic formal manifold structure. The \textbf{$l$-adic normal structure set $\mathbf{S}(Z)^{\wedge}_{N,l}$ of $Z$} consists of equivalent classes of all $l$-adic topological bundle liftings of the mod-$l$ Spivak normal spherical fibration of $Z$.
\end{definition}

For example, if $Z$ itself is a smooth, projective variety over a separably closed field $k$, then an $l$-adic homotopy equivalence map $M\rightarrow Z_{\et}$ with $M$ some other manifold is an element of $\mathbf{S}(Z)^{\wedge}_{N,l}$. In this case, within $\mathbf{S}(Z)_{N,l}^{\wedge}$, there are elements represented by algebraic maps $M\rightarrow Z$ with $M$ another smooth projective variety, which  are called  the \textbf{$k$-algebraic elements in $\mathbf{S}(Z)_{N,l}^{\wedge}$} (see Definition \ref{definition-of-k-algebraic-elements}).

Following the proofs of the Adams conjecture, we construct the \textbf{abelianized Galois action of $\widehat{\ZZ}^{\times}_l$ on $\mathbf{S}(Z)^{\wedge}_{N,l}$} as in Definition \ref{Def: ablianized Galois action}. When $l$ is odd, it is the Adams operations on the real $K$-theory orientations, as suggested in \cite{SullivanMITnotes}*{Chapter 6}; when $l=2$, we use the idea in the first author's thesis (\cite{hu2024revisiting}).

Consider the following group homomorphisms.
In the case when the ground field $k$ has characteristic $p>0$, let $\omega_k$ be the composition $\Gal(k)\rightarrow \Gal(\overline{\FF}_p)=\widehat{\ZZ}\rightarrow \widehat{\ZZ}^{\times}_l$, where the first arrow is the natural quotient map, the second arrow is $x\rightarrow p^x$.
In the case when $k$ has characteristic $0$, let $\omega_{k}$ be the composition $\Gal(k)\rightarrow \widehat{\ZZ}^{\times}=\prod_q \widehat{\ZZ}^{\times}_q\rightarrow \widehat{\ZZ}^{\times}_l$, where the first arrow is the restriction to the roots of unity and the second arrow is the projection map.
We can now state our final result:

\begin{theorem}(\textbf{Galois Symmetry}, Theorem \ref{Galois-action-on-l-adic-normal-structure-set})
Let $k$ be a separably closed field of characteristic $p\geq 0$. Let $Z$ be a smooth, projective variety over $k$ of dimension at least $3$. Assume that $Z^{\wedge}_{\et,l}$ admits an $l$-local lifting and $\pi_1^{\et}(Z)^{\wedge}_l=0$. 

Then the Galois action of $\Gal(k)$ on the $k$-algebraic elements in the $l$-adic normal structure set $\mathbf{S}(Z)^{\wedge}_{N,l}$ factors through the abelianized Galois action of $\widehat{\ZZ}^{\times}_l$ on the entire $l$-adic normal structure set (as in Definition \ref{Def: ablianized Galois action}), via the homomorphism $\omega_k$.
\end{theorem}

The proof of this theorem relies on the implications of the proofs of the Adams conjecture in \cite{Quillen-Adams-Conjecture-1}\cite{Friedlander-Adams-Conjecture}\cite{Sullivan-adams-conjecture} when $l$ is odd.  The case of $l=2$ is more technical, where we need more sophisticated results about topological bundle liftings of a spherical fibrations in the $2$-local case in \cite{MOrgan-Sullivan-surgery}\cite{Brumfiel-Morgan}.

\subsection*{Acknowledgements.}
We would like to thank Piotr Achinger, Bhargav Bhatt, Irina Bobkova, James F. Davis, Mark de Cataldo,  Andres Fernandez Herrero, Ruochuan Liu, James Myer, Longke Tang and Shmuel Weinberger for useful conversations and comments. We are particularly grateful to John Pardon, Dennis Sullivan and Guozhen Wang for many ideas and details in this paper. 
This material is based
upon work supported by the National Science Foundation under Grant No. DMS-1926686. The first author is partially supported by the Simons Foundation International and
NSF Grant 2247322 during this research.
The second author is grateful for the excellent working environment provided by IAS while he was working on this project.
\tableofcontents

\subsection*{Terminologies and Notations}
\begin{itemize}
    \item A pro-space in this paper is a pro-object in the category of pointed, connected CW complexes. But in some statements, we also use Artin-Mazur's notion of pro-space, i.e., a pro-object in the homotopy category of pointed, connected CW complexes. We will make it clear whenever we use Artin-Mazur's notion.
    \item The \'etale homotopy type of a pointed, locally Noetherian scheme $X$ is a pro-space defined by Friedlander \cite{Friedlander-etale-homotopy} or equivalently by Barnea-Schlank \cite{Barnea-Schlank-Etale-Homotopy-Type}; for the \'etale homotopy type of a stack we use the definition by Chough \cite{Chough-Etale-Homotopy-Algebraic-Stacks}.
    \item For a CW complex $Z$, the space $Z_l$ is Bousfield's $l$-completion \cite{Bousfield-localization-of-spaces}; the space $Z^{\wedge}_l$ is Sullivan's $l$-profinite completion \cite{Sullivan-adams-conjecture}. 
    For simply-connected, finite type (i.e., each homotopy group is a finitely generated abelian group) CW complexes, Bousfield's $l$-completion, Bousfield-Kan's $\ZZ/l$-localiztion \cite{Bousfield-Kan} and Sullivan's $l$-profinite completion are equivalent, see \cite[\S2.3]{Barthel-Bousfield-Comparison-of-p-completion}.
    
    \item For a pro-space $Z=\{Z_i\}$, $Z^{\wedge}_l$ is the pro-space by Artin-Mazur's $l$-profinite completion \cite{artin-mazur-etale-homotopy}.
     Note that when the pro-space is given by a constant space $Z$, there is a clash of notation for $Z^{\wedge}_l$: it sometimes denotes Artin-Mazur's pro-space completion and sometimes denotes Sullivan's $l$-profinite space completion. They differ by taking homotopy limit functor (e.g., one way to define this is to use the compact Brownian functor formalism in \cite{Sullivan-adams-conjecture}). We will make the context clear so that no ambiguity is possible in the paper.
    \item For a pro-space $Z=\{Z_i\}$, $\holim Z$ is the homotopy limit of $\{Z_i\}$.
    \item For a (pro-)group $G$, $G^{\wedge}_l$ is the $l$-profinite completion of $G$. For an abelian group $A$, $A_l$ is the $l$-completion of $A$, which is the inverse limit of $A/l^nA$ for all $n$. If $A$ is a finitely generated abelian group, $A^{\wedge}_l$ is canonical isomorphic to $A_l$. $L_0$ and $L_1$ are the left derived functors of the $l$-completion of abelian groups (see \cite{MayMoreConcise}*{p.~193-194}).
    \item The pro-category of finite groups is naturally equivalent to the category of profintie groups, i.e. the compact Hausdorff totally disconnected groups, see e.g. \cite[Prop. 3.2.12]{Lurie-rational-p-profinite}. Therefore, we can and will safely confuse pro systems of finite groups with profinite groups.
    \item Let $E$ be a spectrum. $L_E$ is the $E_*$-theory localization functor of spaces or spectra.
    \item For an $l$-adic spherical fibration $\gamma:Z\rightarrow BG(n)_l$, $S(\gamma)$ is the total space, $D(\gamma)$ is the mapping cylinder of $\gamma:S(\gamma)\rightarrow Z$ and $M(\gamma)$ is the mapping cone of $\gamma$.
    \item If $k$ is a field of characteristic $p>0$, then $W(k)$ is the ring of $p$-typical Witt vectors.
    \item $S^n$ is the space $n$-sphere. $\Sigma^n$ is the n-sphere spectrum and $\Sigma^n_l$
    is the $l$-completion of $\Sigma^n$.
\end{itemize}

The following definition for pro-spaces is needed.

\begin{definition}[\cite{artin-mazur-etale-homotopy}*{Theorem 4.3}, \cite{morel1993quelques}*{Theorem 2.4.1}]\label{definition-of-l-adic-weak-equivalence}
A pro-map $f:X\rightarrow Y$ of pro-spaces is an $l$-adic weak equivalence if one of the following equivalent conditions is satisfied:
\begin{enumerate}
    \item $\widehat{f}_l:\widehat{X}_l\rightarrow \widehat{Y}_l$ induces an isomorphism on pro homotopy groups;
    \item $\holim\widehat{f}_l:\holim\widehat{X}_l\rightarrow \holim\widehat{Y}_l$ is a homotopy equivalence of CW complexes;
    \item $f_*:\pi_1(X)^{\wedge}_l\rightarrow \pi_1(Y)^{\wedge}_l$ is an isomorphism and for any local system of finite abelian $l$-group $M$ over $Y$, $f$ induces an isomorphism $H^*(Y;M)\rightarrow H^*(X;M)$.
    \item $f$ induces an isomorphism $H^*((Y)^{\wedge}_l;\ZZ/l)\rightarrow H^*((X)^{\wedge}_l;\ZZ/l)$.
\end{enumerate}
\end{definition}

\section{Homotopical Finiteness of Simply-Connected Varieties}\label{Section-Finiteness-of-Varieties} 

For a complex variety $X$ with its analytification $X_{\CC}^{an}$, Artin-Mazur's comparison theorem \cite[Theorem 12.9]{artin-mazur-etale-homotopy} states that the \'etale homotopy type  $X_{\et}$ is homotopy equivalent to the profinite completion of $X_{\CC}^{an}$.

In the case where $X$ is instead a scheme over a separably closed field of prime characteristic $p$, there are several recent advances on the fundamental groups:
when $X$ is proper and connected, it is shown in \cite[Theorem 1.2]{lara2024fundamental} that $\pi_1^{\et}(X)$ is topologically finitely presented.
In contrast, it is shown in \cite[Theorem C]{esnault2021obstruction} that there exists a smooth projective non-liftable $X$ such that $\pi_1^{\et}(X)$ is not the profinite completion of a finitely presented discrete group. 

For higher homotopy groups in prime characteristic, we are aware of the following results.
The classical result is the specialization isomorphism in \cite[Corollaries 12.12, 12.13]{artin-mazur-etale-homotopy}, which entails that if $X$ has a smooth proper lift $X_{\CC}$ to characteristic 0, then the homotopy groups $\pi_i(X^{\wedge}_{\et,l})$ and $\pi_i((X^{an}_{\CC})^{\wedge}_l)$ are isomorphic, where $X_{\et}$ is the \'etale homotopy type of $X$, $(-)^{\wedge}_l$ is the $l$-profinite completion of (pro)-spaces and $X^{an}_{\CC}$ is the analytification of $X_{\CC}$ .
On the other hand, when $X$ is affine and connected, then \cite[Theorem 1.1.1]{achinger2017wild} entails that $X$ is a $K(\pi,1)$ space.

The results above leave open the case of higher homotopy groups of non-affine not necessarily liftable varieties, which is the case we study here.
In general, higher homotopy groups of varieties are not finitely generated in any sense. 
For example, over $\CC$, the topological $\pi_2$ of the one-point union of $\mathbb{P}^1$ and $\mathbb{G}_m$ is a direct sum of infinite copies of $\ZZ$, which is in fact a rank one free $\ZZ[\pi_1]$-module. However, its topological $\pi_3$ is not even a finitely generated $\ZZ[\pi_1]$-module.
The problem in this example is that the fundamental group in this case is non trivial.
Indeed, by Serre's mod-$\mathcal{C}$ Hurewicz theorem \cite{Serre-mod-C-Hurewicz}, any simply-connected space has finitely generated homologies if and only if it has finitely generated homotopy groups.

The mod-$\mathcal{C}$ argument cannot work for pro-spaces such as $X_{\et,l}^{\wedge}$,
whose homotopy groups are finitely generated $\widehat{\ZZ}_l$-modules, because such modules do not satisfy Serre's axioms ($\text{II}_A$) and (III) as in \cite{Serre-mod-C-Hurewicz}. However, Theorem \ref{finiteness-theorem-1} still holds.

Recall the following definition for $l$-complete finite type.

\begin{definition}
A simply-connected CW complex or a pro space $Y$ is \textbf{$l$-complete finite type} if each homotopy group $\pi_q(Y)$ is a finitely generated $\widehat{\ZZ}_l$-module.
\end{definition}

\begin{theorem}\label{finiteness-theorem-1}
Let $X$ be a pointed, connected, locally Noetherian scheme and $l$ be a prime number. Assume that $\pi^{\et}_1(X)^{\wedge}_l=0$. Then the $l$-profinite completion $X^{\wedge}_{\et,l}$ of the \'etale homotopy type $X_{\et}$ of $X$ is $l$-complete finite type if and only if the \'etale cohomology $H^q_{\et}(X;\ZZ/l)$ is finite for any $q$.
\end{theorem}

Note that the terminology ``finite type'' in algebraic geometry often refers to the finite dimensional property of a variety. This terminology is also used in topology, where finite type for a CW complexes means that each homotopy group is a finitely generated abelian group.

For some topologists, whether a space is homotopy equivalent to a finite CW complex might be more interesting. The previous theorem can improved as follows.

\begin{theorem}\label{finiteness-theorem-11}
Let $X$ be a pointed, connected, locally Noetherian scheme and $l$ be a prime number. Assume that $\pi^{\et}_1(X)^{\wedge}_l=0$. Then 
\begin{enumerate}
    \item the \'etale homotopy type $X_{\et}$ of $X$ is $l$-adic weak equivalent to a simply-connected CW complex with finitely many cells in each dimension if and only if the \'etale cohomology $H^q_{\et}(X;\ZZ/l)$ is finite for any $q$ and the $l$-profinite completion $X^{\wedge}_{\et,l}$ of the \'etale homotopy type $X_{\et}$ admits an $l$-local lifting (see Definition \ref{l-local-lifting-definition});
    \item $X_{\et}$ is $l$-adic weak equivalent to a simply-connected finite CW complex of dimension at most $N$ if and only if the previous conditions hold and $H^q_{\et}(X;\ZZ/l)=0$ for $q>N$.
\end{enumerate}
\end{theorem}

For a more algebraic discussion of the condition ``$l$-local lifting'', see Remark \ref{l-local-lifting-remark}.

\begin{corollary}\label{finiteness-thereom-on-variety} 
Let $X$ be a variety over a separably closed field $k$ of characteristic $p\geq 0$. Let $l$ be a prime number. Assume that $(\pi_1^{\et}X)^{\wedge}_l=0$.
Assume either of the following is true:
\begin{enumerate}
    \item[(a)] $p\neq l$;
    \item[(b)] $X$ is proper over $k$.
\end{enumerate}
Then 
\begin{enumerate}
    \item $X^{\wedge}_{\et,l}$ is $l$-complete finite type;
    \item $X_{\et}$ is $l$-adic weak equivalent to a simply-connected finite CW complex if and only if the $l$-profinite completion of $X_{\et}$ admits an $l$-local lifting.
\end{enumerate}
\end{corollary}
\begin{proof}
\cite[VI, Theorem 1.1]{Milne-Etale-cohomology} entails that $H^q(X;\mathbb{Z}/l)=0$ for $l> 2 \mathrm{dim}(X)$.
It remains to show that the group $H^q(X,\mathbb{Z}/l)$ is finite.
In the case (a), this follows from  \cite[Th. finitude, Corollary 1.10]{deligne1977sga4.5}.
In the case (b), this follows from
\cite[VI, Cor 2.8]{Milne-Etale-cohomology}.
\end{proof}

Then we provide a homotopical obstruction for characteristic zero lifting of a simply-connected positive characteristic variety.

\begin{corollary}\label{characteristic-0-lifting-homotopically}
Let $X$ be a variety over a separably closed field $k$ of characteristic $p$. Assume that either $l\neq p$ or $X$ is complete.
Assume that $\pi^{\et}_1(X)^{\wedge}_l=0$. Then $X_{\et}$ is $l$-adic weak equivalent to a simply-connected complex variety if and only if the $l$-profinite completion of $X_{\et}$ admits an $l$-local lifting. 
\end{corollary}

\begin{proof}
The `only if' part is obvious. For the `if' part, we use the construction in \cite{Deligne-Sullivan}*{paragraph below Lemma on p.~1082}. $X_{\et}$ is $l$-adic weak equivalent to a simply-connected finite CW complex. Every finite CW complex is homotopy equivalent to a finite polyhedron. Embed this polyhedron piecewise linearly into $\RR^N$ for some large $N$. Complexify the affine subspaces spanned by all simplices of the polyhedron in $\RR^N$. We obtain a finite union of complex affine subspaces in $\CC^N$, which is homotopic to the polyhedron.
\end{proof}

\begin{corollary}\label{characteristic-0-lifting-obstruction}
Let $X$ be a smooth, proper variety over a separably closed field $k$ of characteristic $p>0$. 
\begin{enumerate}
    \item If $\pi^{\et}_1(X)^{\wedge}_l=0$ for a prime $l\neq p$, then a necessary condition for $X$ to lift to a characteristic zero with finite fundamental group is that the $l$-profinite completion of $X_{\et}$ admits an $l$-local lifting. 
    \item If $\pi^{\et}_1(X)^{\wedge}_l=0$ for any prime $l\neq p$, then a necessary condition for $X$ to lift to a characteristic zero with finite fundamental group is that the $l$-profinite completion of $X_{\et}$ admits an $l$-local lifting for any prime $l\neq p$. 
\end{enumerate}
\end{corollary}

\begin{proof}
Let $Y$ be the characteristic zero lifting of $X$ such that $\pi_1(Y_{\CC})$ is finite. Since $X$ is $l$-adic weak equivalent to $Y_{\CC}$, $X^{\wedge}_{\et,l}$ is also $l$-adic weak equivalent to the universal cover $Y_{\CC}$. Then the rest follows from Theorem \ref{finite-CW-complex-1}.
\end{proof}

\begin{remark}
The requirement for the finiteness condition the fundamental group for the lifting in item (2) is needed, because we do not know whether there exists an infinite, finitely presented group $G$ such that the $l$-profinite completion of $G$ is trivial for any prime $l\neq p$. 
\end{remark}

The proof of Theorem \ref{finiteness-theorem-1} relies on the following topological result.

\begin{theorem}\label{finite-type-theorem}
Let $\{Y_i\}$ be a pro system of objects in the homotopy category of pointed, connected CW complexes such that $\pi_q(Y_i)$ is a finite $l$-group for any $q,i$. Assume that $\varprojlim_i \pi_1(Y_i)=0$. Then $\varprojlim_i H_q(Y_i;\ZZ/l)$ is finite for any $q$ if and only if $\holim_{i} Y_i$ is $l$-complete finite type.
\end{theorem}

\begin{proof}[Proof of Theorem \ref{finiteness-theorem-1}]
Since $\varinjlim H^q(X_{\et};\ZZ/l)\cong H^q_{\et}(X;\ZZ/l)$ is finite, by Pontryagin duality, $\varprojlim H_q(X_{\et};\ZZ/l)$ is also finite. 
Then the theorem follows from Corollary \ref{cor of finite} by setting $X_{\et}=\{Y_i\}$.
\end{proof}

Theorem \ref{finiteness-theorem-11} is a direct corollary of Theorem \ref{finite-CW-complex-1} and Corollary \ref{cor of finite} below.

Recall that a simply-connected space $Y$ is $l$-local if each $\pi_i(Y)$ is a $\ZZ_{(l)}$-module, where $\ZZ_{(l)}$ is the $l$-localization of $\ZZ$ (see \cite{SullivanMITnotes}*{p.~32, Theorem 2.2}). This is equivalent to saying that each reduced $\widetilde{H}_i(Y;\ZZ)$ is a $\ZZ_{(l)}$-module.

\begin{definition}\label{l-local-lifting-definition}
A simply-connected, $l$-complete finite type CW complex $Y$ admits an \textbf{$l$-local lifting} if there exists a simply-connected, $l$-local finite type CW complex $Y'$ such that the $l$-completion of $Y'$ is homotopy equivalent to $Y$.
\end{definition}

\begin{theorem}\label{finite-CW-complex-1}
Let $Y$ be a simply-connected CW complex. 
\begin{enumerate}
    \item There exists a simply-connected CW complex $Z$ with finitely many cells in each dimension such that $Z$ is $l$-adic weak equivalent to $Y$ if and only if the $l$-completion $Y_l$ of $Y$ is $l$-complete finite type and admits an $l$-local lifting.
    \item 
    The space $Z$ above can be made a finite CW complex of dimension at most $N$ if and only if the previous conditions hold and $H^q(Y_l; \ZZ/l)=0$ for all $q>N$.
\end{enumerate}
\end{theorem}

\begin{remark}\label{l-local-lifting-remark}
Recall the arithmetic square of rings (see \cite{SullivanMITnotes}*{p.~87})
\[
\begin{tikzcd}
  \ZZ_{(l)} \arrow[r] \arrow[d] & \QQ \arrow[d] \\  
  \widehat{\ZZ}_{l} \arrow[r] &  \QQ_l
\end{tikzcd}
\]
and its homotopy analogue for $X$ a simply-connected, finite type CW complex:
\[
\begin{tikzcd}[row sep=huge,column sep=width("rationalization")]
  X_{(l)} \arrow[r,"\text{rationalization}"] \arrow[d,"\text{$l$-completion}"] & X_{\QQ} \arrow[d,"\text{$l$-formal completion}"] \\  
  \widehat{X}_{l} \arrow[r,"\text{rationalization}"] &  X_{\QQ_l}.
\end{tikzcd}
\]
The arithmetic square entails that whether a simply-connected, $l$-complete finite type CW complex $Y=Y_l$ admits an $l$-local lifting is equivalent to whether the rationalization $Y_{\QQ}=Y_{\QQ_l}$ admits a rational finite type space lifting. Our conjecture is that this might be equivalent to whether the minimal model of $Y_{\QQ_l}$ in differential graded algebras over $\QQ_l$ can be lifted to $\QQ$ (see \cite{SullivanMITnotes}*{p.~88}).
\end{remark}

We will prove Theorem \ref{finite-type-theorem} in Section \ref{sec: proof of finite-type-theorem} and Theorem \ref{finite-CW-complex-1} in Section \ref{section: pf of finite-cw-complex-1}.

\begin{corollary}
\label{cor of finite}
With the same assumptions as Theorem \ref{finite-type-theorem}, 
\begin{enumerate}
    \item $\{Y_i\}$ is $l$-adic weak equivalent to a simply-connected CW complex $Z$ with finitely many cells in each dimension if and only if $\varinjlim_i H_q(Y_i;\ZZ/l)$ is finite for any $q$ and $\holim_{i} Y_i$ admits an $l$-local lifting;
    \item Furthermore, $Z$ above can be a finite CW complex of dimension at most $N$ if and only if the previous conditions and $\varinjlim_i H^q(Y_i;\ZZ/l)=0$ for all $q>N$.
\end{enumerate}
\end{corollary}

\begin{proof}
The `only if' parts of the two statements are both obviously. For the  `if' parts, let $Y=\holim_i Y_i$. 
By Theorem \ref{finite-type-theorem}, $Y$ satisfies the assumptions in Theorem \ref{finite-CW-complex-1}.  
Let $Z$ and $Z_l$ be as in Theorem \ref{finite-CW-complex-1}.
Since $\pi_q(Z)\rightarrow\pi_q(Z_l)$ is the $l$-profinite completion of a finitely generated abelian group for any $q$, $Z_l\rightarrow Y$ is a weak equivalence. Hence, the induced pro map $Z\rightarrow Z_l\rightarrow Y\rightarrow \{Y_i\}$ is an $l$-adic weak equivalence by definition.
\end{proof}

\subsection{Proof of Theorem \ref{finite-type-theorem}}\label{sec: proof of finite-type-theorem}

\begin{lemma}\label{Eilenberg-Maclane-space}
Let $\{A_i\}$ be a pro system of finite abelian $l$-groups. Let $A=\varprojlim A_i$. Assume that $A$ is a finitely generated $\widehat{\ZZ}_l$-module. Then $H_q(K(A,n);\ZZ/l)\rightarrow \varprojlim_i H_q(K(A_i,n);\ZZ/l)$ is an isomorphism for any $q,n$. In particular, $\varprojlim_i H_q(K(A_i,n);\ZZ/l)$ is finite.
\end{lemma}

\begin{proof}
Since $A$ is a finitely generated $\widehat{\ZZ}_l$-module, $K(A,n)$ is homotopy equivalent to its $l$-completion. Therefore,
$\pi_q(K(A,n)_l)\cong\pi_q(K(A,n))\cong \varprojlim \pi_q(K(A_i,n))$.
It then follows from \cite{artin-mazur-etale-homotopy}*{Theorem 4.3} that $H^q(K(A,n);\ZZ/l)\leftarrow \varinjlim_i H^q(K(A_i,n);\ZZ/l)$ is an isomorphism. 
The lemma follows from the Pontryagin duality.
\end{proof}

\begin{lemma}\label{Hurewicz-type-lemma}
Let $\{Y_i\}$ be a pro system of objects in the homotopy category of pointed, connected CW complexes with $\pi_q(Y_i)$ a finite $l$-group for any $q,i$. For any $i$ assume that $\pi_q(Y_i)=0$ for $1\leq q\leq n-1$ with $n\geq 2$, and that $\varprojlim_i H_n(Y_i;\ZZ/l)$ is finite. Then $\varprojlim_i \pi_n(Y_i)$ is a finitely generated $\widehat{\ZZ}_l$-module.
\end{lemma}

\begin{proof}
By Hurewicz theorem, $H_n(Y_i;\ZZ)\cong \pi_n(Y_i)$ is a finite abelian $l$-group. Then the following map is an isomorphism.
\[
\pi_n(Y_i)/l\cdot \pi_n(Y_i)\cong H_n(Y_i;\ZZ)/l\cdot H_n(Y_i;\ZZ)\cong H_n(Y_i;\ZZ)\otimes \ZZ/l\rightarrow H_n(Y_i;\ZZ/l)
\]
Thus $\varprojlim_i (\pi_n(Y_i)/l\cdot \pi_n(Y_i))\rightarrow \varprojlim_i H_n(Y_i;\ZZ/l)$ is also an isomorphism. But $\varprojlim_i (\pi_n(Y_i)/l\cdot \pi_n(Y_i))\cong \varprojlim_i \pi_n(Y_i)/ \varprojlim_il\cdot \pi_n(Y_i)\cong \varprojlim_i \pi_n(Y_i)/ l\cdot\varprojlim_i \pi_n(Y_i)$. 
Since an $l$-profinite abelian group is $l$-complete, by \cite{stacks-project}*{Tag 0315}, $\varprojlim_i \pi_n(Y_i)$ is a finitely generated $\widehat{\ZZ}_l$-module. 
\end{proof}

\begin{lemma}\label{inverse-limit-commutes-with-coefficient}
Let $Y$ be a CW complex homotopy equivalent to a CW complex with finitely many cells in each dimension. Let $\{A_i\}$ be a pro-system of finite abelian groups. Then the canonical map $H_q(Y;\varprojlim_i A_i) \rightarrow \varprojlim_i H_q(Y;A_i)$ is an isomorphism for any $q$.
\end{lemma}

\begin{proof}
This is obvious when $Y$ is a sphere.
Let $Y_n$ be the $n$-skeleton of $Y$.
The lemma follows by induction using the cofiber sequence $\bigsqcup_j S^n_j\rightarrow Y_n\rightarrow Y_{n+1}$.
\end{proof}

\begin{lemma}\label{Serre-spectral-sequence}
Let $\{F_i\rightarrow E_i\rightarrow B_i\}_{i\in I}$ be a pro-system of fibrations in the homotopy category of pointed, connected CW complexes with $\pi_q(F_i),\pi_q(E_i),\pi_q(B_i)$ finite $l$-groups for any $q,i$. Assume that each $B_i$ is simply-connected. If $\varprojlim_i H_q(F_i;\ZZ/l)$ and $\varprojlim_i H_q(B_i;\ZZ/l)$ are finite for any $q$, then so is $\varprojlim_i H_q(E_i;\ZZ/l)$.
\end{lemma}

\begin{proof}

Consider the pro-system of Leray-Serre spectral sequence $E(i)^r_{p,q}$ for the fibration $F_i\rightarrow E_i\rightarrow B_i$ with $E(i)^2_{p,q}=H_p(B_i;H_q(F_i;\ZZ/l))$. Take the inverse limit $\varprojlim_i E(i)^r_{p,q}$ for fixed $r,p,q$. This is still a spectral sequence since each $E(i)^r_{p,q}$ is finite.
\begin{align*}
    \varprojlim_i E(i)^2_{p,q} & = \varprojlim_i H_p(B_i;H_q(F_i;\ZZ/l)) \cong \varprojlim_{i,j} H_p(B_i;H_q(F_j;\ZZ/l))  \\
    & = \varprojlim_{i}\varprojlim_{j} H_p(B_i;H_q(F_j;\ZZ/l)) \cong \varprojlim_{i}H_p(B_i;\varprojlim_{j} H_q(F_j;\ZZ/l)) \,\,\, \text{(Lemma \ref{inverse-limit-commutes-with-coefficient})} \\ 
\end{align*}
By the assumptions, $\varprojlim_i E(i)^2_{p,q}$ is finite for any $p,q$. $\varprojlim_i E(i)^{\infty}_{p,q}$ is finite for any $p,q$ as well.

Recall that the filtration of the homology groups $0=F_{-1}H_n\subset F_0H_n\subset F_1H_n\subset ... \subset H_n$ are given by the skeletons of the base space. So any map between fibrations preserve this filtration. Hence, 
\begin{align*}
\varprojlim_i E(i)^{\infty}_{p,q} & 
 =\varprojlim_i \Big(F_pH_{p+q}(E_i;\ZZ/l)/F_{p-1}H_{p+q}(E_i;\ZZ/l) \Big)\\ 
  & \cong \Big(\varprojlim_i F_pH_{p+q}(E_i;\ZZ/l)\Big)/\Big(\varprojlim_iF_{p-1}H_{p+q}(E_i;\ZZ/l)\Big) \\
  & \,\,\, \text{(finite colimits commute with cofiltered limits)} \\ 
  & \cong  F_p\varprojlim_i H_{p+q}(E_i;\ZZ/l)/F_{p-1}\varprojlim_iH_{p+q}(E_i;\ZZ/l) \\
  & \,\,\, \text{(maps preserve the filtration).} 
\end{align*}
Since for each $n$ the length of the filtration on $H_n$ is finite, this proves the lemma.
\end{proof}





\begin{proof}[Proof of Theorem \ref{finite-type-theorem}]
Let $\widetilde{Y_i}$ be the universal cover of $Y_i$.
Since $\varprojlim_i \pi_1(Y_i)=0$, the induced map $\holim_i \widetilde{Y_i}\rightarrow \holim_i Y_i$ is a homotopy equivalence, thus $\varinjlim_i H^q(\widetilde{Y_i};\ZZ/l)\leftarrow \varinjlim_i H^q(Y_i;\ZZ/l)$ is an isomorphism.
By Pontryagin duality, $\varprojlim_i H_q(\widetilde{Y_i};\ZZ/l)\rightarrow \varprojlim_i H_q(Y_i;\ZZ/l)$ is an isomorphism for all $q$.
Therefore, by replacing $Y_i$ with $\widetilde{Y_i}$, we can assume that each $Y_i$ is simply-connected. 

The `if' part can be directly deduced from Lemma \ref{Serre-spectral-sequence} and the Postnikov towers of $Y_i$'s. 

For the `only if' part, let $Y^{(n)}_i$ be the $n$-th space in the Postnikov tower of $Y_i$. Let $F^{(n+1)}_i$ be the homotopy fiber of $Y\rightarrow Y^{(n)}_i$. By Lemma \ref{Hurewicz-type-lemma}, $\varprojlim_i \pi_2(Y_i)$ is a finitely generated $\widehat{\ZZ}_l$-module. By Lemma \ref{Eilenberg-Maclane-space} and Lemma \ref{Serre-spectral-sequence}, $\varprojlim_iH_q(F^{(3)}_i;\ZZ/l)$ is finite for any $q$. Inductively assume that $\varprojlim_i \pi_q(Y_i)$ is a finitely generated $\widehat{\ZZ}_l$-module for any $q\leq n$ and that $\varprojlim_iH_q(F^{(n+1)}_i;\ZZ/l)$ is finite for any $q$. By Lemma \ref{Hurewicz-type-lemma}, $\varprojlim_i \pi_{n+1}(Y_i)\cong \varprojlim_i \pi_{n+1}(F^{(n+1)}_i)$ is a finitely generated $\widehat{\ZZ}_l$-module. Applying Lemma \ref{Eilenberg-Maclane-space} and Lemma \ref{Serre-spectral-sequence} to the fibration $F^{(n+2)}_i\rightarrow F^{(n+1)}_i\rightarrow K(\pi_{n+1}(Y_i),n+1)$, we have that $\varprojlim_iH_q(F^{(n+2)}_i;\ZZ/l)$ is finite for any $q$. Then the proof is completed by induction.
\end{proof}

\subsection{Proof of Theorem \ref{finite-CW-complex-1}}\label{section: pf of finite-cw-complex-1}\;

In this part, we prove Theorem \ref{finite-CW-complex-1}. 
The proof of Theorem \ref{finite-CW-complex-1} relies on the following result.

\begin{theorem}\label{thm: when l-local is finite}
Let $Y$ be a connected, simply-connected, $l$-local CW complex.  
\begin{enumerate}
    \item Each $\widetilde{H}_i(Y;\ZZ)$ is a finitely generated $\ZZ_{(l)}$-module if and only if there exists a simply-connected CW complex $Z$ with finitely many cells in each dimension such that its $l$-localization $Z_{(l)}$ is homotopy equivalent to $Y$;
    \item Each $\widetilde{H}_i(Y;\ZZ)$ is a finitely generated $\ZZ_{(l)}$-module and $H^q(Y;\ZZ)=0$ for all $q>N$ if and only if we can take the $Z$ above to be a finite CW complex of dimension at most $N$.
\end{enumerate}
\end{theorem}

\begin{remark}
This theorem can also be proved by local cells and local CW complexes (\cite{SullivanMITnotes}*{Definition 2.2 and p.~86}).
\end{remark}

\begin{proof}
(1) The `only if' part is obvious. We prove the `if' part here. We will inductively construct a simply-connected CW complex $Z$ with finitely many cells together with a map $f:Z\rightarrow Y$ which induces an isomorphism $f_{(l)}:\pi_*(Z)\otimes \ZZ_{(l)}\cong \pi_*(Z_{(l)})\rightarrow \pi_*(Y)$. 

Firstly, by the Hurewicz theorem, $\pi_2(Y)\cong H_2(Y;\ZZ)$ is a finitely generated $\ZZ_{(l)}$-module. Let  $x_1,\dots,x_m$ be $\ZZ_{(l)}$-module generators of $\pi_2(Y)$. Then construct a map $Z_2\rightarrow Y$ with $Z_2$ a bouquet of $m$ copies of $S^2$ such that the mapping on each $S^2$ represents $x_1,\dots,x_m$ correspondingly. Then this induces a surjection on $\pi_2(-)\otimes Z_{(l)}$.

Inductively assume that we have constructed a simply-connected finite CW complex $Z_n$ of dimension $n$ together with a map $f_n:Z_n\rightarrow Y$ so that the relative homotopy group $\pi_i((f_n)_{(l)})$ of the $l$-localization $(f_n)_l:(Z_n)_{(l)}\rightarrow Y$ vanishes for all $i\leq n$. Consider the following diagram.

\[
\begin{tikzcd}
  ... \arrow[r] & \pi_{i}(Z_n) \arrow[r,"(f_n)_*"] \arrow[d] & \pi_{i}(Y) \arrow[r] \arrow[d,] & \pi_{i}(f_n) \arrow[r] \arrow[d] & ... \\
  ... \arrow[r] & \pi_{i}(Z_n)\otimes \ZZ_{(l)} \arrow[r,"(f_n)_{*}\otimes \ZZ_{(l)}"] \arrow[d,"\cong"] & \pi_{i}(Y)\otimes \ZZ_{(l)} \arrow[r] \arrow[d,"\cong"] & \pi_{i}(f_n)\otimes \ZZ_{(l)} \arrow[r] \arrow[d,dotted,"\cong"] & ... \\  
  ... \arrow[r] & \pi_{i}((Z_n)_{(l)}) \arrow[r,"(f_n)_{(l),*}"] & \pi_{i}(Y) \arrow[r] & \pi_{i}((f_n)_{(l)}) \arrow[r] &  ...
\end{tikzcd}
\]

Since localization is an exact functor, the second row is exact. By the universal property of $l$-localization of abelian groups, the natural map $\pi_i(f)\rightarrow \pi_{i}((f_n)_{(l)})$ factors through $\pi_{i}(f_n)\otimes \ZZ_{(l)}$, which induces the dashed arrow. Then the five-lemma deduces that the dashed arrow is an isomorphism.

Since $H_i(Y;\ZZ)$ and $H_i((Z_n)_{(l)};\ZZ)$ are both finitely generated $\ZZ_{(l)}$-modules for any $i$, so is $H_i(f_{(l)};\ZZ)$, in particular, $H_{n+1}((f_n)_{(l)};\ZZ)\cong \pi_{n+1}((f_{n})_{(l)})$. So one may choose $y_1,\dots,y_r$ in $\pi_{n+1}(f_n)$ whose images generate $\pi_{n+1}((f_n)_{(l)})$ as a $\ZZ_{(l)}$-module. Each $y_i$ corresponds to a commutative diagram
\[
\begin{tikzcd}
S^n \arrow[r,"\alpha_i"] \arrow[d] & Z_n \arrow[d,"f"] \\
D^{n+1} \arrow[r,"\beta_i"] & Y
\end{tikzcd}
\]

Then attach cells $e^{n+1}_1,...,e^{n+1}_r$ to $Z_n$, whose boundaries are $\alpha_i(S^n)$ correspondingly. We also obtain a map $f_{n+1}:Z_{n+1}\rightarrow Y$ to be the gluing of $f_n$ and $\beta_i$'s.

It is left to check that $\pi_i((f_{n+1})_{(l)})=0$ for $i\leq n+1$. But this is obvious by considering the following long exact sequence induced by the Hurewicz theorem and the composition of maps $(f_n)_{(l)}:(Z_n)_{(l)}\xrightarrow{j_{(l)}} (Z_{n+1})_{(l)} \xrightarrow{(f_{n+1})_{(l)}} Y_l$:
\[
...\rightarrow H_i(j_{(l)};\ZZ)\rightarrow H_i((f_n)_{(l)};\ZZ)\rightarrow H_i((f_{n+1})_{(l)};\ZZ)\rightarrow ...
\]

(2) As above, the `only if' part is also obvious so we only prove the `if' part here. By the inductive construction previously, we can construct a simply-connected finite CW complex $Z_{N-1}$ of dimension $(N-1)$ together with a map $f_{N-1}:Z_{N-1}\rightarrow Y$ such that $\pi_i((f_{N-1})_{(l)})$ vanishes for $i\leq N-1$. By the Hurewicz theorem, $\pi_N((f_{N-1})_{(l)})\rightarrow H_N((f_{N-1})_{(l)};\ZZ)$ is an isomorphism. 

It is obvious that $\pi_N((f_{N-1})_{(l)})\cong H_N((f_{N-1})_{(l)};\ZZ)$ is a finitely generated $\ZZ_{(l)}$-module. Consider the long exact sequence
\[
...\rightarrow H_N((Z_{N-1})_{(l)};\ZZ)=0\rightarrow H_N(Y;\ZZ)\rightarrow H_N((f_{N-1})_{(l)};\ZZ)\rightarrow H_{N-1}((Z_{N-1})_{(l)};\ZZ)\rightarrow ...
\]
$H_N(Y;\ZZ)$ is a finitely generated free $\ZZ_{(l)}$-module by the univeresal coefficient theorem and the assumption that $H^q(Y;\ZZ)=0$ for $q>N$. Obviously, the same holds for $H_{N-1}((Z_{N-1})_{(l)};\ZZ)$. Then so is $H_N((f_{N-1})_{(l)};\ZZ)$. 

Let $y_1,...,y_r\in \pi_N((f_{N-1})_{(l)})\cong H_N((f_{N-1})_{(l)};\ZZ)$ be a basis over $\ZZ_{(l)}$. Then one can scale $y_1,...,y_r$ by products of primes not equal to $l$ so that there exists $z_1,...,z_r\in\pi_N(f_{N-1})$ such that the images of $z_i$'s are $y_i$'s correspondingly.

Then attach cells to $Z_{N-1}$ to kill $z_1,...,z_r$ and we obtain a map $f_N:Z_N\rightarrow Y$. Consider the long exact sequence of homology for the triple $((Z_{N-1})_{(l)},(Z_N)_{(l)},Y)$, it follows that $H_q(Y,(Z_N)_{(l)};\ZZ)=0$ for any $q$. Therefore, the inductive procedure stops at $Z_N$.
\end{proof}

\begin{proof}[Proof of Theorem \ref{finite-CW-complex-1}]
For the only if direction, we use the fact that the $l$-completion the $l$-localization of a simply-connected finite type CW complex $Z$ is the $l$-completion of $Z$.
For the if direction, given a simply connected CW complex $Y$ as in Theorem \ref{finite-CW-complex-1}, we apply Theorem \ref{thm: when l-local is finite} to the $l$-local complex whose $l$-completion is $Y_l$. 
\end{proof}

\subsection{Limits and Homologies}\label{sec: homological lemma}\;

We record the following lemma, to be used in the next sections, which gives a sufficient condition for when (homotopy) limit commute with homology. The general case is complicated (see \cite{Goerss-homology-homotopy-inverse-limit}).

\begin{lemma}\label{commutativity-of-limit-and-homology}
Let $\{Y_i\}$ be a pro system of objects in the homotopy category of pointed, connected CW complexes such that $\pi_q(Y_i)$ is a finite $l$-group for all $q,i$. Let $Y=\holim_i Y_i$. Assume that $\varprojlim_i \pi_1(Y_i)=0$ and $\varprojlim_i H_q(Y_i;\ZZ/l)$ is finite for any $q$. 
Let $A$ be a finite abelian $l$-group.

Then the canonical maps $H_q(Y;A)\rightarrow \varprojlim_i H_q(Y_i;A)$ and $H^q(Y;A)\leftarrow \varinjlim_i H^q(Y_i;A)$ are isomorphisms. 
\end{lemma}

\begin{corollary}\label{equivalence-of-l-completion-and-l-profinite-completion}
Let $Z$ be a simply-connected CW complex. Assume that $H_q(Z;\ZZ/l)$ is finite for any $q$. Then its $l$-profinite completion $Z^{\wedge}_l$ and its $l$-completion $Z_l$ are identical. In particular, $Z^{\wedge}_l$ is $l$-complete finite type.
\end{corollary}
\begin{proof}
Let the pro-space $\{Z_j\}$ be the Artin-Mazur's $l$-profinite completion of $Z$. Then $Z^{\wedge}_l\simeq \holim_j Z_j$ by definition. By \cite{artin-mazur-etale-homotopy}*{Theorem 4.3}, the canonical map $H^q(Z;\ZZ/l)\leftarrow \varinjlim_j H^q(Z_j;\ZZ/l)$ is an isomorphism for any $q$. By Pontryagin duality, $\varprojlim_j H_q(Z_j;\ZZ/l)$ is finite. By Lemma \ref{commutativity-of-limit-and-homology}, $H_q(Z^{\wedge}_l;\ZZ/l)\rightarrow \varprojlim_j H_q(Z_i;\ZZ/l)\cong H_q(Z;\ZZ/l)$ is an isomorphism. By Theorem \ref{finite-type-theorem}, $Z^{\wedge}_l$ is $l$-complete finite type. The isomorphism $H_q(Z;\ZZ/l)\cong H_q(Z^{\wedge}_l;\ZZ/l)$ implies that $Z^{\wedge}_l$ is the $l$-completion of $Z$.
\end{proof}

\begin{lemma}\label{Serre-spectral-sequence-2}
Let $\{F_i\rightarrow E_i\rightarrow B_i\}_{i\in I}$ be a pro-system of fibrations in the homotopy category of pointed, connected CW complexes such that $\pi_q(F_i),\pi_q(E_i),\pi_q(B_i)$ are finite $l$-groups for any $q,i$. 
Let $F,E,B$ be the homotopy inverse limit of $\{F_i\},\{E_i\},\{B_i\}$ respectively. Assume that each $B_i$ is simply-connected. Further assume that $\varprojlim_i H_q(F_i;\ZZ/l)$ and $\varprojlim_i H_q(B_i;\ZZ/l)$ are finite for any $q$. If $H_q(F;\ZZ/l)\rightarrow \varprojlim_i H_q(F_i;\ZZ/l)$ and $H_q(B;\ZZ/l)\rightarrow \varprojlim_i H_q(B_i;\ZZ/l)$ are isomorphisms for any $q$, then $H_q(E;\ZZ/l)\rightarrow \varprojlim_i H_q(E_i;\ZZ/l)$ is also an isomorphism for any $q$. 
\end{lemma}

\begin{proof}
Let $\widetilde{F}$ be the homotopy fiber of $E\rightarrow B$. Then we have a map $\widetilde{F}\rightarrow F_i$, which is compatible with respect to the pro-system $\{F_i\}$.
The map $\widetilde{F}\rightarrow \{F_i\}$ induces a map $\widetilde{F}\rightarrow F$. Consider the following commutative diagram. 
\[
\begin{tikzcd}
 ... \arrow[r] & \pi_q(\widetilde{F}) \arrow[r] \arrow[d] & \pi_q(E) \arrow[r] \arrow[d,phantom,sloped,"\cong"] &  \pi_q(B) \arrow[r] \arrow[d,phantom,sloped,"\cong"] & ... \\
 ... \arrow[r] & \varprojlim_i\pi_q(F_i) \arrow[r] & \varprojlim_i\pi_q(E_i) \arrow[r] &  \varprojlim_i\pi_q(B_i) \arrow[r] & ...    
\end{tikzcd}
\]
The indicated arrows are isomorphisms by definition. The second row is exact because it is exact before taking inverse limit and each term is finite abelian. By the five-lemma, $\pi_q(F)\cong \varprojlim_i\pi_q(F_i)$. So $\widetilde{F}\rightarrow F$ is a homotopy equivalence.

Consider the pro-system of Leray-Serre spectral sequence $E(i)^r_{p,q}$ for the fibration $F_i\rightarrow E_i\rightarrow B_i$ with $E(i)^2_{p,q}=H_p(B_i;H_q(F_i;\ZZ/l))$ and the Leray-Serre spectral sequence $E^r_{p,q}$ for the fibration $F\rightarrow E\rightarrow B$ with $E^2_{p,q}=H_p(B;H_q(F;\ZZ/l))$. The inverse limit $\varprojlim_i E(i)^r_{p,q}$ is also a spectral sequence. Because of the following equations, $E^{\infty}_{p,q}\cong \varprojlim_i E(i)^{\infty}_{p,q}$.
\begin{align*}
    \varprojlim_i E(i)^2_{p,q} & = \varprojlim_i H_p(B_i;H_q(F_i;\ZZ/l)) \\
    & \cong \varprojlim_{i}H_p(B_i;\varprojlim_{j} H_q(F_j;\ZZ/l)) \,\,\, \text{(Lemma \ref{inverse-limit-commutes-with-coefficient})} \\ 
    & \cong \varprojlim_{i}H_p(B_i;H_q(F;\ZZ/l)) \cong H_p(B;H_q(F;\ZZ/l)) = E^2_{p,q}  \\
\end{align*}
Consider the filtration of the homology groups $0=F_{-1}H_n\subset F_0H_n\subset F_1H_n\subset ... \subset H_n$.
\begin{align*}
F_pH_{p+q}(E;\ZZ/l)/F_{p-1}H_{p+q}(E;\ZZ/l) & = E^{\infty}_{p,q} \cong \varprojlim_i E(i)^{\infty}_{p,q} 
 =\varprojlim_i F_pH_{p+q}(E_i;\ZZ/l)/F_{p-1}H_{p+q}(E_i;\ZZ/l) \\ 
  & \cong  F_p\varprojlim_i H_{p+q}(E_i;\ZZ/l)/F_{p-1}\varprojlim_iH_{p+q}(E_i;\ZZ/l) \\
\end{align*}
Since for each $n$ the length of the filtraion on $H_n$ is finite, this proves the lemma.
\end{proof}

\begin{proof}[Proof of Lemma \ref{commutativity-of-limit-and-homology}]
By the Bockstein exact sequence, we can reduce to the case $A=\ZZ/l$.
By Pontryagin duality, the cohomological isomorphism in Lemma \ref{commutativity-of-limit-and-homology} follows from the homological isomorphism. By a similar argument as in the proof of Theorem \ref{finite-type-theorem}, we may assume that each $Y_i$ is simply-connected without loss of generality. By Theorem \ref{finite-type-theorem}, $Y$ is simply-connected, $l$-complete finite type.

Let $Y_i^{(n)},Y^{(n)}$ be the $n$-th spaces in the Postnikov tower of $Y_i,Y$ respectively. By Lemma \ref{Eilenberg-Maclane-space} and Lemma \ref{Serre-spectral-sequence-2}, $H_q(Y^{(n)};\ZZ/l)\rightarrow \varprojlim_i H_q(Y_i^{(n)};\ZZ/l)$ is an isomorphism for any $q$.
Since $Y^{(n)}$ is obtained by attaching cells of dimension larger than $n+1$, $H_q(Y^{(n)};\ZZ/l)\cong H_q(Y;\ZZ/l)$ for any $q<n$. The same is true for $Y_i$. So for any fixed $q$ and some large enough $n$, $H_q(Y;\ZZ/l)\cong H_q(Y^{(n)};\ZZ/l)\cong \varprojlim_i H_q(Y_i^{(n)};\ZZ/l) \cong \varprojlim_i H_q(Y_i;\ZZ/l)$.
\end{proof}

\section{$l$-adic Poincar\'e Duality}\label{Section-Foundations-of-l-adic-Spivak}

In this section, we develop the theories of $l$-adic Poincar\'e duality spaces, $l$-adic Spanier-Whitehead duality, and mod-$l$ Spivak normal spherical fibrations. 
On the first pass, one may only read Definition \ref{definition-l-adic-spherical-fibration}, \ref{definition-mod-l-Poincare-duality}, \ref{definition-of-mod-l-spivak}, \ref{definition-of-equivalence-of-spivak} and Proposition/Theorem \ref{Z/l-poincare-to-spivak}, \ref{uniqueness-of-l-adic-Spivak} and skip the details in Sections \ref{subsection-l-adic-s-duality}-\ref{subsec: normal spherical}.

In Section \ref{subsection-mod-l-poincare-duality-on-varieties}, we briefly review Joshua's mod-$l$ Spanier-Whitehead duality for projective varieties (\cite{Joshua-Spanier-Whitehead}) and prove the existence of topological mod-$l$ Poincar\'e duality for the \'etale homotopy type of a smooth projective variety (see Theorem \ref{existence-of-mod-Poincare-duality}).

\subsection{$l$-adic Spanier-Whitehead Duality}\label{subsection-l-adic-s-duality}
\;

There are two versions of $l$-completion of spectra: the $S(\ZZ/l)$-localization and the $H(\ZZ/l)$-localization, where $S(\ZZ/l)$ is the Moore spectrum and $H(\ZZ/l)$ is the Eilenberg-Maclane spectrum. These two $l$-completions are equivalent for connective spectra (see \cite{Bousfield-localiztion-of-spectra}*{Theorem 3.1}).

\begin{proposition}\label{l-complete-wedge-with-finite-spectra}
If a spectrum $V$ is $l$-complete, then $V\wedge Z$ is $l$-complete for any finite spectrum $Z$.    
\end{proposition}

\begin{proof}
$V\wedge Z$ is obviously $l$-complete when $Z=\Sigma^n$ for any $n\in \ZZ$. 
Suppose that $V\wedge Y$ is $l$-complete for a finite spectrum $Y$ and that there is a cofibration sequence $\Sigma^n\rightarrow Y\rightarrow Z$.
The long exact sequence of the $V_*(-)$-homologies deduces that $V_*(Z)$ is derived $l$-complete.
Since $V\wedge Z$ is bounded below, \cite{Barthel-Bousfield-Comparison-of-p-completion}*{Corollary 3.3} entails that $V\wedge Z$ is $l$-complete. 
\end{proof}

\begin{corollary}\label{cor of l-complete wedge}
If $X,Y$ are connective spectra, then $(X\wedge Y)_l\rightarrow (X\wedge Y_l)_l$ is a homotopy equivalence.   
\end{corollary}

\begin{proof}
Consider the following commutative diagram.
\[
\begin{tikzcd}
H_*(X;\ZZ/l)\otimes H_*(Y;\ZZ/l) \arrow{r}{\cong} \arrow{d}{\cong} &  H_*(X;\ZZ/l)\otimes H_*(Y_l;\ZZ/l) \arrow{d}{\cong} \\
H_*(X\wedge Y;\ZZ/l) \arrow{r} & H_*(X\wedge Y_l,\ZZ/l)
\end{tikzcd}
\]
It follows that $(X\wedge Y)_l\simeq L_{H(\ZZ/l)}(X\wedge Y)\rightarrow L_{H(\ZZ/l)}(X\wedge Y_l)\simeq (X\wedge Y_l)_l$.
\end{proof}

Let $\Sigma^n$ be the $n$-sphere spectrum and let $\Sigma^n_l$ be the $l$-completion of $\Sigma^n$.

Let $Z,Z^*$ be two connective spectra with a map $\mu:Z^*\wedge Z\rightarrow \Sigma^0_l$. 
Let $U$ be an $l$-complete spectrum and $V$ be a connective $l$-complete spectrum  with maps $f:U\rightarrow (V\wedge Z^*)_l$ and $g: U\to (Z\wedge V)_l$. 
Then map $\mu_l$ induces two maps:
\[
D_{\mu}(f):U\wedge Z \xrightarrow{(f\wedge Id_Z)_l} ((V\wedge Z^*)_l\wedge Z)_l\xleftarrow{\simeq} (V\wedge Z^*\wedge Z)_l \xrightarrow{(Id_V\wedge \mu)_l} (V\wedge \Sigma^0_l)_l\simeq (V\wedge \Sigma^0)_l \simeq V,
\]
\[
{}_{\mu}D(g):Z^*\wedge U\xrightarrow{(Id_{Z^*}\wedge g)_l} (Z^*\wedge (Z\wedge V)_l)_l \xleftarrow{\simeq} (Z^*\wedge Z\wedge V)_l \xrightarrow{(\mu\wedge Id_V)_l} (\Sigma^0_l\wedge V)_l\simeq (\Sigma^0\wedge V)\simeq V.
\]
In this way, we obtain two maps $D_{\mu}:[U,(V\wedge Z^*)_l]\rightarrow [U\wedge Z,V]$ and ${}_{\mu}D:[U,(Z\wedge V)_l]\rightarrow [Z^*\wedge U,V]$, which are natural in $U$ and $V$.

\begin{definition}\label{definition-of-l-adic-S-duality}
$\mu:(Z^*\wedge Z)_l\rightarrow \Sigma^0_l$ is called an \textbf{$l$-adic Spanier-Whitehead duality} (or an \textbf{$l$-adic $S$-duality}) with $Z^*,Z$ connective spectra if $D_{\mu}$ and ${}_{\mu}D$ are isomorphisms for any $l$-complete spectrum $U$ and any connective $l$-complete spectrum $V$. 
In this case $Z^*$ is called an \textbf{$l$-adic Spanier-Whitehead dual} (or an \textbf{$l$-adic $S$-dual}) of $Z$.
\end{definition}

\begin{proposition}\label{equivalence-of-S-duality-and-l-adic-S-duality}
Let $Z^*$ and $Z$ be finite spectra with $S$-duality $\mu:Z^*\wedge Z\rightarrow \Sigma^0$.
Then $\mu_l:Z^*\wedge Z\rightarrow (Z^*\wedge Z)_l\rightarrow \Sigma^0_l$ is an $l$-adic $S$-duality.    
\end{proposition}

\begin{proof}
Let $U$ be an $l$-complete spectrum and let $V$ be a connective $l$-complete spectrum. By Proposition \ref{l-complete-wedge-with-finite-spectra}, $V\wedge Z^*$ and $V\wedge Z$ are both $l$-complete. The following commutative diagram shows that $D_{\mu_l}$ is an isomorphism.
\[
\begin{tikzcd}
    \left [ U,V\wedge Z^* \right ] \arrow{rr}{\cong} \arrow{rd}{D_{\mu,\cong}} & & \left [ U,(V\wedge Z^*)_l \right ] \arrow{ld}{D_{\mu_l}} \\
    & \left [ U\wedge Z,V \right] & 
\end{tikzcd}
\]
The isomorphy of ${}_{\mu_l}D$ is shown similarly.
\end{proof}

\begin{remark}
Let $\tau:Z^*\wedge Z\rightarrow Z\wedge Z^*$ be the map interchanging the two factors. Then there is a commutative diagram.
\[
\begin{tikzcd}
\left [U,(V\wedge Z^*)_l\right ] \arrow{r}{D_{\mu}} \arrow{d}{\cong} & \left [U\wedge Z,V\right ] \arrow{d}{\cong} \\
\left [U,(Z^*\wedge V)_l\right ] \arrow{r}{{}_{\mu\tau}D} & \left [Z\wedge U,V\right ]
\end{tikzcd}
\]
It follows that
$\mu$ is an $l$-adic $S$-duality if and only if so is $\mu\circ\tau$.
\end{remark}

Let $\mu:Z^*\wedge Z\rightarrow \Sigma^0$ and $\nu:Y^*\wedge Y\rightarrow \Sigma^0$ be $l$-adic $S$-duality maps. The composition
\[\theta(\mu,\nu):[Z_l,Y_l]\cong [Z_l,(Y\wedge \Sigma^0_l)_l]\underset{\sim}{\overset{{{}_{\nu}D}}{\rightarrow}}[Y^*\wedge Z_l,\Sigma^0_l] \cong [(Y^*)_l\wedge Z,\Sigma^0_l]\underset{\sim}{\overset{{D_{\mu}}}{\leftarrow}} [(Y^*)_l,(Z^*\wedge \Sigma^0_l)_l]\cong [(Y^*)_l,(Z^*)_l]\]
is an isomorphism. For $f:Z_l\rightarrow Y_l$, define $f^*=\theta(\mu,\nu)(f):(Y^*)_l\rightarrow (Z^*)_l$. 
Note that the homotopy class of $f^*$ is uniquely characterized by the following diagram, commutative up to homotopy:
\[
\begin{tikzcd}
    (Y^*)_l\wedge Z_l \arrow[r,"f^*\wedge Id_{Z_l}"] \arrow[d,"Id_{Y^*}\wedge f"] & (Z^*)_l\wedge Z_l \arrow[d,"\mu_l"] \\
    (Y^*)_l\wedge Y_l \arrow[r,"\nu_l"] & \Sigma^0_l
\end{tikzcd}
\]  
for which we have used the isomorphism $[Z\wedge Z^*,\Sigma^0_l]\cong [(Z\wedge Z^*)_l,\Sigma^0_l]\cong [(Z_l\wedge (Z^*)_l)_l,\Sigma^0_l]\cong [Z_l\wedge (Z^*)_l,\Sigma^0_l]$ and a similar isomorphism for $Y,Y^*$.

The following proposition is obvious.

\begin{proposition}
Let $\mu:Z^*\wedge Z\rightarrow \Sigma^0_l$, $\nu:Y^*\wedge Y\rightarrow \Sigma^0_l$ and $\pi:Z^*\wedge Z\rightarrow \Sigma^0_l$ be $l$-adic S-dualities for connective spectra. For $f:Z_l\rightarrow Y_l$ and $g:Y_l\rightarrow Z_l$,
\begin{enumerate}
    \item $(f^*)^*=f$;
    \item $(g\circ f)^*=f^*\circ g^*$;
    \item $Id^*=Id$.
\end{enumerate}
\end{proposition}

\begin{proposition}
If $\mu:Z^*\wedge Z\rightarrow \Sigma^0_l$ and $\widetilde{\mu}:\widetilde{Z^*}\wedge Z\rightarrow \Sigma^0_l$ are both $S$-dualities, then there is a homotopy equivalence $h:Z^*_l\rightarrow (\widetilde{Z^*})_l$, unique up to homotopy, such that $\widetilde{\mu}_l\circ(h\wedge Id_{Z_l})_l\simeq \mu_l$.    
\end{proposition}

\begin{proof}
$Id_{Z_l}:Z_l\rightarrow Z_l$ induces $h=Id_{Z_l}^*:(Z^*)_l\rightarrow (\widetilde{Z^*})_l$ and $g=Id_{Z_l}^*:(\widetilde{Z^*})_l\rightarrow (Z^*)_l$, unique up to homotopy, such that $\widetilde{\mu}_l\circ(h\wedge Id_{Z_l})_l\simeq \mu_l$ and $\mu_l\circ(g\wedge Id_{Z_l})_l\simeq \widetilde{\mu}_l$. Now consider $h\circ g$ and $Id_{(\widetilde{Z^*})_l}$. They are both $Id^*_{Z_l}:(\widetilde{Z^*})_l\rightarrow (\widetilde{Z^*})_l$ so they are homotopic. Similarly, $g\circ h\simeq Id_{(Z^*)_l}$.
\end{proof}

\begin{example}
The $l$-adic $S$-duality of $\Sigma^n$ is $\Sigma^{-n}$.
\end{example}

A spectrum $Z$ is said to be \textbf{$l$-adic homotopically finite} if there is a finite spectrum $X$ such that $X_l\simeq Z_l$. 
Let $V=H(\ZZ/l)$. Using Prop. \ref{l-complete-wedge-with-finite-spectra} and Corollary \ref{cor of l-complete wedge}, $(V\wedge Z)_l\simeq (V\wedge Z_l)_l \simeq (V\wedge X_l)_l\simeq (V\wedge X)_l \simeq V\wedge X$.
It follows that $H_n(X;\ZZ/l)\cong H_n(Z_l;\ZZ/l)\cong H_n(Z;\ZZ/l)$.

\begin{theorem}\label{homological-discription-of-S-duality-1}
Let $Z$ and $Z^*$ be $l$-adic homotopically finite spectra. Then a map $\mu:Z^*\wedge Z\rightarrow \Sigma^0_l$ is an $l$-adic $S$-duality if and only if $D_{\mu}=\mu^*([\Sigma^0_l]_{\ZZ/l}^{\vee})/(-):H_*((Z^*)_l;\ZZ/l)\rightarrow H^{-*}(Z;\ZZ/l)\cong H^{-*}(Z_l;\ZZ/l)$ is an isomorphism, where $[\Sigma^0]_{\ZZ/l}^{\vee}\in H^0(\Sigma^0;\ZZ/l)$ is a generator, $\mu^*$ is the pullback on cohomologies, and $/(-)$ is the slant product.
\end{theorem}

\begin{proof}
We assume that $D_{\mu}=\mu^*([\Sigma^0_l]_{\ZZ/l}^{\vee})/(-):H_*((Z^*)_l;\ZZ/l)\rightarrow H^{-*}(Z;\ZZ/l)\cong H^{-*}(Z_l;\ZZ/l)$ is an isomorphism since the other direction is obvious. 

 Since $l$-adic $S$-duality is invariant under $l$-completion, we may assume that both $Z^*$ and $Z$ are finite spectra. 
 There is a map $\theta'(\mu,\mu'):[Z_l,Z_l]\xrightarrow{{}_{\mu}D}[(Z^*\wedge Z)_l,\Sigma^0_l]\cong [(Z^*)_l\wedge Z,\Sigma^0_l]\xrightarrow{D_{\mu'},\cong} [(Z^*)_l,Z'\wedge \Sigma^0_l]\cong [(Z^*)_l,(Z')_l]$. Let $h:(Z^*)_l\rightarrow (Z')_l$ be $\theta(\mu,\nu)(Id_{Z_l})$. Then $h$ is uniquely characterized by the following diagram which commutes up to homotopy. 
\[
\begin{tikzcd}
    (Z^*)_l\wedge Z_l \arrow{rr}{h\wedge Id_{Z_l}} \arrow{rd}{\mu_l} & & (Z')_l\wedge Z_l \arrow{ld}{\mu'_l} \\
    & \Sigma^0_l &
\end{tikzcd}
\]
Hence, we have the following commutative diagram
\[
\begin{tikzcd}
    H_q((Z')_l;\ZZ/l) \arrow{rr}{D_{\mu'},\cong} & & H^{-q}(Z_l;\ZZ/l) \\
    & H_q((Z^*)_l;\ZZ/l) \arrow{lu}{h_*} \arrow{ru}{D_{\mu}} &
\end{tikzcd}
\]
Since $D_{\mu}$ is an isomorphism, $h_*$ is an isomorphism for any $q$. Then $h:(Z^*)_l\rightarrow (Z')_l$ is an $H(\ZZ/l)$-equivalence. Since $Z^*,Z$ are both connective, $(Z^*)_l,(Z')_l$ are both $H(\ZZ/l)$-local. Thus, $h$ is a homotopy equivalence. The upper commutative diagram proves that $\mu$ is an $l$-adic $S$-duality.
\end{proof}

For $S$-dualities, there is an alternative way to define it and these two ways are equivalent. We study this phenomenon in the $l$-complete sense.

Again let $Z,Z^*$ be connective spectra with a map $\rho:\Sigma^0\rightarrow (Z\wedge Z^*)_l$. There are two maps for any $l$-complete spectrum $U$ and connective $l$-complete spectrum $V$, natural in both $U,V$.
\[
D^{\rho}:[U\wedge Z^*,V]\rightarrow [U,(V\wedge Z)_l],{}^{\rho}D:[Z\wedge U,V]\rightarrow [U,(Z^*\wedge V)_l]
\]
They are defined in a similar way as $D_{\mu},{}_{\mu}D$.

\begin{lemma}
If $\mu:Z^*\wedge Z\rightarrow \Sigma^0_l$ and $\nu:Y^*\wedge Y\rightarrow \Sigma^0_l$ are $l$-adic $S$-dualities, then so is $(\mu,\nu): Z^*\wedge Y^*\wedge Y\wedge Z\xrightarrow{Id_{Z^*}\wedge \nu\wedge Id_Z} Z^*\wedge \Sigma^0_l\wedge Z\rightarrow (Z^*\wedge \Sigma^0_l\wedge Z)_l\simeq (Z^*\wedge Z)_l\xrightarrow{\mu_l} \Sigma^0_l$.
\end{lemma}

\begin{proof} 
The lemma follows from the following commutative diagram:
\[
\begin{tikzcd}
\left [ U,(V\wedge Z^*\wedge Y^*)_l\right ]\cong  \left [ U,((V\wedge Z^*)_l\wedge Y^*)_l\right ]\arrow{rr}{D_{(\mu,\nu)}} \arrow{rd}{D_\nu,\cong} & & \left [ U\wedge Y\wedge Z,V\right ] \\
& \left [ U\wedge Y,(V\wedge Z^*)_l\right ] \arrow{ru}{D_{\mu},\cong}&
\end{tikzcd}
\]
\end{proof}

Let $\mu:Z^*\wedge Z\rightarrow \Sigma^0_l$ and $\nu:Y^*\wedge Y\rightarrow \Sigma^0_l$ be $l$-adic $S$-duality maps. For any $\rho^*: Y\wedge Z\rightarrow \Sigma^0_l$, by Proposition \ref{equivalence-of-S-duality-and-l-adic-S-duality} we have the $S$-dual map $\rho:\Sigma^0\rightarrow \Sigma_l^0 \rightarrow (Z^*\wedge Y^*)_l$.  $\rho$ is uniquely characterized by the following diagram which commutes up to homotopy.
\[
\begin{tikzcd}
    S^0\wedge Y\wedge Z \arrow{r}{(\rho\wedge Id_{Y\wedge Z})_l} \arrow{d}{(Id_{\Sigma^0}\wedge \rho^*)_l} & ((Z^*\wedge Y^*)_l\wedge Y\wedge Z)_l\simeq (Z^*\wedge Y^*\wedge Y\wedge Z)_l \arrow{d}{(\mu,\nu)_l} \\
    (\Sigma^0_l\wedge \Sigma^0)_l\simeq (\Sigma^0\wedge \Sigma^0)_l \arrow{r}{\simeq} & \Sigma^0_l
\end{tikzcd}
\]

\begin{lemma}
$\rho^*$ is an $S$-duality if and only if $D^{\rho},{}^{\rho}D$ are isomorphisms for any $l$-complete $U$ and any connective $l$-complete $V$.
\end{lemma}

\begin{proof}
The following commutative diagram proves that $D_{\rho^*}$ is an isomorphism if and only if ${}^{\rho}D$ is.
\[
\begin{tikzcd}
\left [U,(V\wedge Y)_l\right ] \arrow{r}{\cong} \arrow{d}{D_{\rho^*}} & \left [U,(Y\wedge V)_l\right ] \arrow{r}{{}_{\nu}D,\cong} & \left [Y^*\wedge U,V\right ] \arrow{d}{{}^{\rho}D} \\
\left [U\wedge Z,V\right ] & \left [U,(V\wedge Z^*)_l\right ] \arrow{l}{D_{\mu},\cong} & \left [U,(Z^*\wedge V)_l\right ] \arrow{l}{\cong}
\end{tikzcd}
\]    
Similarly, ${}_{\rho^*}D$ is an isomorphism if and only if $D^{\rho}$ is.
\end{proof}

In particular, let $Y^*=Z$, $Y=Z^*$,  $\nu=\mu\circ \tau$ and $\rho^*=\mu$, then we obtain a map $\rho=\mu^*:\Sigma^0\rightarrow (Z^*\wedge Z)_l$. Then this lemma shows that $\mu$ is an $l$-adic $S$-duality if and only if $D^{\rho},{}^{\rho*}D$ are isomorphisms for any $l$-complete $U$ and any connective $l$-complete $V$. 

So we may equivalently define that $\rho=\mu^*:\Sigma^0\rightarrow (Z^*\wedge Z)_l$ is an \textbf{$l$-adic $S$-duality} if $D^{\rho},{}^{\rho*}D$ are isomorphisms for any $l$-complete $U$ and any connective $l$-complete $V$.  

Analogous to the previous discussion, if $\rho:\Sigma^0\rightarrow (Z^*\wedge Z)_l$ and $\psi:\Sigma^0\rightarrow (Y^*\wedge Y)_l$ are both S-dualities, then for any $f:Z_l\rightarrow Y_l$, there corresponds to a $S$-dual map $f^*:(Y^*)_l\rightarrow (Z^*)_l$, which is uniquely characterized by the following diagram which commutes up to homotopy.
\[
\begin{tikzcd}
    \Sigma^0_l \arrow{r}{\rho} \arrow{d}{\psi} & (Z^*\wedge Z)_l\simeq (Z^*\wedge Z_l)_l \arrow{d}{(Id_{Z^*}\wedge f)_l} \\
    (Y^*\wedge Y)_l\simeq ((Y^*)_l\wedge Y) \arrow{r}{(f^*\wedge Id_Y)_l} & ((Z^*)_l\wedge Y)_l\simeq (Z^*\wedge Y)_l\simeq (Z^*\wedge Y_l)_l
\end{tikzcd}
\]

The proof of the following theorem is similar to that of Theorem \ref{homological-discription-of-S-duality-1}.

\begin{theorem}\label{homological-discription-of-S-duality-2}
Let $Z,Z^*$ be $l$-adic homotopically finite spectra. Then a map $\rho:\Sigma^0\rightarrow (Z^*\wedge Z)_l$ is an $l$-adic $S$-duality if and only if $D^{\rho}=(-)\backslash \rho_*([\Sigma^0_l]_{\ZZ/l}):H^*((Z^*)_l;\ZZ/l)\rightarrow H_{-*}(Z;\ZZ/l)\cong H_{-*}(Z_l;\ZZ/l)$ is an isomorphism, where $[\Sigma^0]_{\ZZ/l}\in H_0(\Sigma^0;\ZZ/l)$ is a generator.
\end{theorem}

\subsection{$l$-adic Spherical Fibration}\;

Let $G(n)$ be the topological monoid of self homotopy equivalences of $S^{n-1}$ and let $SG(n)$ be the connected component of $G(n)$ containing the identity. Recall that $BG(n)$ is the classifying space of fiberwise homotopy equivalence classes of $S^{n-1}$-fibrations and $BSG(n)$ is the classifying space of orientation-preserving fiberwise homotopy equivalence classes of oriented $S^{n-1}$-fibrations \cite{Stasheff-spherical-fibrations}.
There are natural inclusions $G(n)\rightarrow G(n+1)$ and let $G=\bigcup_{n} G(n)$. Then $BG$ is the classifying space of stable fiberwise homotopy equivalence classes of stable spherical fibrations.

\begin{definition}[\cite{SullivanMITnotes}*{p.~89, Definition}]\label{definition-l-adic-spherical-fibration}
Let $Z$ be a CW complex. An \textbf{$l$-adic spherical fibration} is a Hurewicz fibration $\gamma: S(\gamma)\rightarrow Z$ with homotopy fiber $(S^{n-1})_l$. 
The \textbf{Thom space $M(\gamma)$} of $\gamma$ is the mapping cone of $S(\gamma)\rightarrow Z$.
An \textbf{orientation} or a \textbf{Thom class} on $\gamma$ is a class $U_{\gamma}\in \widetilde{H}^n(M(\gamma);\widehat{\ZZ}_l)$ whose restriction to each fiber is a generator in $\widetilde{H}^n((S^{n-1})_l;\widehat{\ZZ}_l)$. 
\end{definition}

An orientation $U_{\gamma}\in \widetilde{H}^N(M(\gamma);\widehat{\ZZ}_l)$ induces Thom isomorphisms $H^*(Z;\widehat{\ZZ}_l)\rightarrow \widetilde{H}^{N+*}(M(\gamma);\widehat{\ZZ}_l)$ and $H^*(Z;\widehat{\ZZ}/l)\rightarrow \widetilde{H}^{N+*}(M(\gamma);\ZZ/l)$. 
Below we recall several facts about the classifying space of $l$-adic spherical fibrations (see \cite{SullivanMITnotes}*{Theorem 4.1, Theorem 4.2}).

\begin{fact}
\begin{enumerate}
    \item There exists a classifying space $BG(n)_l$ for the fiberwise homotopy equivalence classes of $(S^{n-1})_l$-fibrations.
    \item The fiberwise $l$-completion of spherical fibrations correspond to a map $BG(n)\rightarrow BG(n)_l$.
    \item $\pi_1(BG(n))\rightarrow \pi_1(BG(n)_l)$ is the homomorphism of units $\{\pm 1\}\cong \ZZ^{\times}\rightarrow \widehat{\ZZ}^{\times}_l$.
    \item The $l$-completion $BSG(n)_l$ of $BSG(n)$ is the classifying space for the orientation-preserving fiberwise homotopy equivalence classes of oriented $(S^{n-1})_l$-fibrations.
    \item The map of universal covers over $BG(n)\rightarrow BG(n)_l$ is the $l$-completion map $BSG(n)\rightarrow BSG(n)_l$.
\end{enumerate}
\end{fact}

\begin{fact}[\cite{SullivanMITnotes}*{Chapter VI, Corollary 1}]
Let $G(n)_l$ be the topological monoid of self homotopy equivlalences of $(S^{n-1})_l$. Then 
$\pi_0(G(n)_l)\cong \widehat{\ZZ}^{\times}_l$ and the connected component of $G(n)_l$ containing the identity is equivalent to the $l$-completion $SG(n)_l$ of $SG(n)$ as H-spaces. 
\end{fact}

\begin{definition}
Since $((S^m)_l*(S^n)_l)_l\simeq (S^{m+n+1})_l$, the \textbf{fiberwise Whitney join $\gamma\oplus \gamma'$} is defined for any $l$-adic spherical fibrations $\gamma,\gamma'$.
\end{definition}

The rest part of this subsection is an $l$-adic version of the dicussions in \cite{Browder-Surgery-Simply-Connected}*{p.~22-23}.

\begin{definition}
For an $l$-adic spherical fibration $\gamma$ over $Z$, let \textbf{$\Aut(\gamma)$} be the group of fiber homotopy classes of fiberwise homotopy equivalences $\gamma\rightarrow \gamma$ which cover identity on $Z$.    
\end{definition}

\begin{definition}
There is a natual map $\Aut(\gamma)\rightarrow \Aut(\gamma\oplus \epsilon^1)$ given by $f\rightarrow f\oplus Id_{\epsilon^1}$. Define $\mathcal{A}(\gamma)$ to be the colimit of $\Aut(\gamma\oplus \epsilon^n)$ for $n\rightarrow \infty$.  
\end{definition}

The involution in $G(n)$ induced by changing orientation defines an involution $\gamma\mapsto -\gamma$ for an oriented $l$-adic spherical fibration $\gamma$.
Note that $f\mapsto f\oplus Id_{-\gamma}: \Aut(\gamma)\rightarrow \Aut(\gamma\oplus (-\gamma))$ defines a homomorphism $H:\mathcal{A}(\gamma)\rightarrow \mathcal{A}(\epsilon)$ and $g\mapsto Id_{\gamma}\oplus g: \Aut(\epsilon^q)\rightarrow \Aut(\gamma\oplus \epsilon^q)$ defines a homomorphism $I:\mathcal{A}(\epsilon)\rightarrow \mathcal{A}(\gamma)$, where $\epsilon$ stands for the stable trivial fibration. Obviously, $HI=Id,IH=Id$. This deduces the following proposition.

\begin{proposition}
There is a canonical equivalence $\mathcal{A}(\gamma)\cong \mathcal{A}(\gamma')$ for any two oriented $l$-adic spherical fibrations $\gamma,\gamma'$ over $Z$.     
\end{proposition}

The following is obvious.

\begin{proposition}
Let $\epsilon^{n}$ be the trivial $(S^{n-1})_l$-fibration over $Z$. Then $\Aut(\epsilon^n)\cong [Z,G(n)_l]$.
\end{proposition}

\begin{proposition}
$\pi_i(G(n+1)_l,G(n)_l)=0$ for $i\leq n-2$. Moreover, $[Z,G(n)_l]\cong [Z,G(n+1)_l]$ if the dimension of $Z$ is less than $n-2$ and $[Z,BG(n)_l]\cong [Z,BG(n+1)_l]$ if the dimension of $Z$ is less than $n-3$.
\end{proposition}

\begin{proof}
$G(n)_l\rightarrow G(n+1)_l$ induces an isomorphism on fundamental group. By Lemma \ref{density-of-relative-completion}, $\pi_i(SG(n+1),SG(n))$ is dense in the finitely generated $\widehat{\ZZ}_l$-module $\pi_i(SG(n+1)_l,SG(n)_l)$. This proposition follows from the fact that $\pi_i(G(n+1),G(n))=0$ for $i\leq n-2$ (\cite{Browder-Surgery-Simply-Connected}*{Proposition I.4.10}).
\end{proof}

\begin{proposition}
$\Aut(\epsilon^n)\cong \mathcal{A}(\epsilon)$ if the dimension of $Z$ is less than $n-2$.
\end{proposition}

\begin{proposition}
$\Aut(\gamma)\cong\Aut(\gamma\oplus \epsilon ^1)$ and hence $\Aut(\gamma)\cong \mathcal{A}(\gamma)$ if the dimension of $Z$ is less than $n-2$.
\end{proposition}

The proof is the same as that of \cite{Browder-Surgery-Simply-Connected}*{Theorem I.4.12} for integral spherical fibrations.

Now assume that $Z$ is a simply-connected finite type CW complex. By the universal property of $l$-completion, an $l$-adic spherical fibration $\gamma:Z\rightarrow BG(N)_l$ factors through $\gamma_l:Z_l\rightarrow BG(N)_l$. 

\begin{lemma}\label{completion-lemma-for-spherical-fibration}
The total space $S(\gamma_l)$ of $\gamma_l$ is the $l$-completion of the total space $S(\gamma)$ of $\gamma$. 
Furthermore, $S(\gamma_l)$ is also the $l$-profinite completion of $S(\gamma)$.
\end{lemma}

\begin{proof}
By the long exact sequence of the homotopy groups for the fibration $\gamma_l$, the homotopy groups of $S(\gamma_l)$ are all finitely generated $\widehat{\ZZ}_l$-modules and hence $S(\gamma_l)$ is $l$-complete finite type. 
It follows that the bundle morphism $F_{\gamma}:S(\gamma)\rightarrow S(\gamma_l)$ induced by the fibration pullback of $f_l$ factors the $l$-completion of $S(\gamma)$, i.e., it factors as $S(\gamma)\xrightarrow{g_l} S(\gamma)_l\xrightarrow{F_l} S(\gamma_l)$. It induces the following commutative diagram.
\[
\begin{tikzcd}
... \arrow[r] & \pi_j((S^{N-1})_l) \arrow[r] \arrow[d,phantom,sloped,"="] & \pi_j(S(\gamma)) \arrow[r] \arrow[d,"(g_l)*"] & \pi_j(Z) \arrow[r] \arrow[d,"(f_l)_*"] & ... \\ 
... \arrow[r] & \pi_j((S^{N-1})_l) \arrow[r] \arrow[d,phantom,sloped,"="] & \pi_j(S(\gamma)_l) \arrow[r] \arrow[d,"F_l"] & \pi_j(Z_l) \arrow[r] \arrow[d,phantom,sloped,"="] & ... \\
... \arrow[r] & \pi_j((S^{N-1})_l) \arrow[r] & \pi_j(S(\gamma_l)) \arrow[r] & \pi_j(Z_l) \arrow[r] & ... 
\end{tikzcd}
\]
The first row and the third row are both exact.
By Leray-Serre spectral sequence, all cohomology groups $H^n(S(\gamma);\ZZ/l)$ are finite. By \cite{MayMoreConcise}*{Corollary 10.1.15}, $L_1\pi_j(S(\gamma))=0$ and $L_0\pi_j(S(\gamma))\cong \pi_j(S(\gamma))_l$. By \cite{MayMoreConcise}*{Theorem 11.1.2}, $\pi_j(S(\gamma)_l)\cong L_0\pi_j(S(\gamma))$ and the same is true for $\pi_j(Z_l)$. \cite{MayMoreConcise}*{Lemma 10.1.8} further implies that the second row is also exact. By five-lemma $F_l$ induces isomorphisms on homotopy groups. Hence, $S(\gamma_l)$ is the $l$-completion of $S(\gamma)$.

Since $H_q(S(\gamma);\ZZ/l)$ is finite for any $q$, Corollary \ref{equivalence-of-l-completion-and-l-profinite-completion} entails that  $S(\gamma_l)$ is the $l$-profinite completion of $S(\gamma)$.
\end{proof}

\begin{corollary}
$M(\gamma)_l$ is both the $l$-completion and the  $l$-profinite completion of $M(\gamma_l)$.    
\end{corollary}

Hence, the Thom class $U_{\gamma_l}$ of $\gamma_l$ is compatible with $U_{\gamma}$ via the bundle map $F_{\gamma}:S(\gamma)\rightarrow S(\gamma_l)$.

Since $((S^m)_l* S^0)_l\simeq (S^{m+1})_l$, this defines a fiberwise degree one map $S(\gamma\oplus \epsilon^1)\rightarrow S((\gamma\oplus \epsilon^1)_l)\simeq S(\gamma\oplus \epsilon^1)_l\rightarrow M(\gamma)_l$, where $\epsilon^1$ is the trivial $S^0$ bundle over $X$. This construction will be used in the next subsection.

\subsection{$l$-adic Poincar\'e Duality and Mod-$l$ Spivak Normal Spherical Fibration}\label{subsec: normal spherical}\;

In this part, we assume that $Z$ is a simply-connected finite CW complex. Let $f_l:Z\rightarrow Z_l$ be the $l$-completion of $Z$.

\begin{definition}
$Z$ is \textbf{$l$-adic Poincar\'e of formal dimension $m$} if there exists a class $[Z]_l\in H_m(Z;\widehat{\ZZ}_l)$ which induces a $\widehat{\ZZ}_l$-coefficient Poincar\'e duality on $Z$, i.e., $(-) \cap [Z]_l:H^{m-*}(Z;\widehat{\ZZ}_l)\rightarrow H_*(Z;\widehat{\ZZ}_l)$ is an isomorphism.
\end{definition}

\begin{definition}\label{definition-mod-l-Poincare-duality}
$Z$ is \textbf{mod-$l$ Poincar\'e of formal dimension $m$} if there exists a class $[Z]_{\ZZ/l}\in H_m(Z;\ZZ/l)$ which induces a mod-$l$ Poincar\'e duality on $Z$.
\end{definition}

\begin{definition}
$Z$ is \textbf{mod-$l$ $l$-complete Poincar\'e of formal dimension $m$} if there exists a class $[Z_l]_{\ZZ/l}\in H_m(Z_l;\ZZ/l)$ which induces a mod-$\ZZ/l$ Poincar\'e duality on $Z_l$.
\end{definition}

Indeed, the three definitions are equivalent.

\begin{proposition}\label{Z/l-Poincare-duality-equivalence}
$Z$ is mod-$l$ Poincar\'e if and only if $Z$ is mod-$l$ $l$-complete Poincar\'e.
\end{proposition}

\begin{proof}
This is obvious by the definition of $l$-completion and the universal coefficient theorem.
\end{proof}

\begin{proposition}\label{Z/l-Poincare-equivalent-to-l-adic-Poincare}
$Z$ is mod-$l$ Poincar\'e if and only if $Z$ is $l$-adic Poincar\'e.
\end{proposition}

\begin{proof}
The `if' part is obvious. Assume $Z$ is $l$-adic Poincar\'e. By the cohomological universal coefficient theorem (see \cite[p.~246, Theorem 10]{Spanier-Algebraic-Topology}) and the usual universal coefficient theorem, we have the following commutative diagram.
\[
\begin{tikzcd}
 0 \arrow[r] & H^*(Z;\widehat{\ZZ}_l)\otimes_{\widehat{\ZZ}_l} \ZZ/l \arrow[r] \arrow[d] & H^*(Z;\ZZ/l) \arrow[r] \arrow[d] & \Tor^{\widehat{\ZZ}_l}(H^{*+1}(Z;\widehat{\ZZ}_l),\ZZ/l) \arrow[r] \arrow[d] &  0   \\
 0 \arrow[r] & H_{m-*}(Z;\widehat{\ZZ}_l)\otimes_{\widehat{\ZZ}_l} \ZZ/l \arrow[r] & H_{m-*}(Z;\ZZ/l) \arrow[r] & \Tor^{\widehat{\ZZ}_l}(H_{m-*-1}(Z;\widehat{\ZZ}_l),\ZZ/l) \arrow[r] &  0
\end{tikzcd}
\]
The rows are exact and the vertical arrows are induced by capping with $[Z]_l$. The vertical arrows on the two sides are isomorphisms by assumption. Therefore, the middle arrow is an isomorphism.

For the `only if' part, assume that $Z$ is mod-$l$ Poincar\'e. By universal coefficient theorem, $H_{m}(Z;\widehat{\ZZ}_l)\cong \widehat{\ZZ}_l$. Choose a lifting $x\in H_{m}(Z;\widehat{\ZZ}_l)\cong \widehat{\ZZ}_l$ whose mod-$l$ reduction is $[Z]_{\ZZ/l}$.
We claim that $x$ induces a $\widehat{\ZZ}_l$-coefficient Poincar\'e duality on $Z$.

Let $C_*$ be the cellular chain complex of $Z$ with $\widehat{\ZZ}_l$-coefficient and let $D_*$ be $\homo_{\widehat{\ZZ}_l}(C_*,\widehat{\ZZ}_l)$ shifted by degree $m$. The $-\cap x$ induces a chain map $D_*\rightarrow C_*$ of bounded finitely generated free $\widehat{\ZZ}_l$-chain complexes. 
Consider the Bockstein spectral sequence (see \cite[Example 24.2.3]{MayMoreConcise}) for both $C_*$ and $D_*$.
The mod-$l$ Poincar\'e duality on $Z$ induces an isomorphism on the $E^1$-page so it induces an isomorphism on all pages of the spectral sequences. By the relation between the order of torsions and the Bockstein spectral sequence (see the last two paragraphs in \cite[p.~481]{MayMoreConcise}), $(-)\cap x$ induces a quasi-isomorphism $D_*\rightarrow C_*$, which proves the proposition.
\end{proof}

\begin{definition}\label{definition-of-mod-l-spivak}
If $Z$ is mod-$l$ Poincar\'e with respect to the fundamental class $[Z]_{\ZZ/l}\in H_m(Z;\ZZ/l)$, then
a \textbf{mod-$l$ Spivak normal spherical fibration over $Z$} consists of an oriented $l$-adic spherical fibration $Z\xrightarrow{\gamma} BSG(N)_l$ together with a map $\phi:S^{N+m}\rightarrow M(\gamma)_l$ such that 
$\phi_*[S^{N+m}]\cap U_{\gamma_l}=(f_l)_*[Z]_{\ZZ/l}$ in $H_m(Z_l;\ZZ/l)$.
\end{definition}

\begin{lemma}
Let $F\rightarrow E\xrightarrow{\pi} Z$ be a fibration with $Z,F$ simply-connected. Then $F_l\simeq (S^{N-1})_l$ if and only if there exists $U\in H^{N}(\pi;\ZZ/l)$ such that $(-)\cup U:H^q(Z;\ZZ/l)\rightarrow H^{q+N}(\pi;\ZZ/l)$ is an isomorphism for any $q$. 
\end{lemma}

\begin{proof}
 The proof is exactly the same as that of \cite{Browder-Surgery-Simply-Connected}*{Lemma I.4.3}, except the following claim: if we have that $F$ is simply-connected, that  $H^q(F;\ZZ/l)=0$ for any $q\neq N-1$, and that $H^{N-1}(F;\ZZ/l)\cong\ZZ/l$, then $F_l\simeq (S^{N-1})_l$. By universal coefficient theorem, $H_q(F;\ZZ/l)=0$ for any $q>0$ except $H_{N-1}(F;\ZZ/l)\cong \ZZ/l$. By Corollary \ref{equivalence-of-l-completion-and-l-profinite-completion}, $F_l$ is $l$-complete finite type. By Hurewicz theorem and universal coefficient theorem, $\pi_{N-1}(F)/l\cdot \pi_{N-1}(F)\cong H_{N-1}(F;\ZZ/l) \cong \ZZ/l$. Choose a map $g:S^{N-1}\rightarrow F_l$ which is a generator of $\ZZ/l$. Then $g_*:H_q(S^{N-1};\ZZ/l)\rightarrow H_q(F_l;\ZZ/l)$ is an isomorphism for any $q$ and hence $g$ induces an $l$-complete equivalence.
\end{proof}

\begin{proposition}\label{Z/l-poincare-to-spivak}
If $Z$ is mod-$l$ Poincar\'e with respect to the fundamental class $[Z]_{\ZZ/l}\in H_m(Z;\ZZ/l)$, then there exists a mod-$l$ Spivak normal spherical fibration over $Z$.
\end{proposition}

\begin{remark}
In general this Spivak normal spherical fibration
is not the fiberwise $\ZZ/l$-localization of some (integral) spherical fibration.    
\end{remark}

\begin{proof}
Up to homotopy, we may assume that $Z$ is a finite simplicial complex. Then embed $Z$ piecewise-linearly in a Euclidean space $\RR^{m+N}$ with $N$ large enough. Let $W$ be a regular neighborhood of $Z$. Then $W$ is a smooth manifold with boundary $\partial W$. $W$ is homotopy equivalent to $Z$. Consider the following two Poincar\'e dualities.
\[
(-) \cap [W,\partial W]_{l}: H^{m+N-*}(W,\partial W;\widehat{\ZZ}_l)\xrightarrow{\cong} H_{*}(W;\widehat{\ZZ}_l)
\]
\[
(-) \cap [Z]_{\ZZ/l}:H^{m-*}(Z;\ZZ/l)\xrightarrow{\cong} H_*(Z;\ZZ/l)\cong H_*(W;\ZZ/l)
\]
Hence, there exists $U\in H^N(W,\partial W ;\ZZ/l)$ such that $U\cap [W,\partial W]_{\ZZ/l}=[Z]_{\ZZ/l}$ and $(-) \cup U: H^*(W;\ZZ/l)\rightarrow H^{*+N}(W,\partial W;\ZZ/l)$
is an isomorphism.

Then the homotopy fiber $F$ of the map $\partial W\rightarrow W\simeq Z$ is a simply-connected, finite type CW complex and $F_l\simeq (S^{N-1})_l$. Fiberwise $l$-complete the fibration replacement of the map $\partial W\rightarrow Z$, we get a $(S^{N-1})_l$-fibration $\gamma$ over $Z$. Moreover, $(W/\partial W)^{\wedge}_l$ is the $l$-completion of the Thom space $M(\gamma)$.

The conclusion that the map $\phi:S^{N+m}\rightarrow W/\partial W\rightarrow (W/\partial W)^{\wedge}_l$ induces $\phi_*[S^{N+m}]\cap U_{\gamma}=(f_l)_*[Z]_{\ZZ/l}$ follows from the mod-$l$ Poincar\'e duality of $Z$ and the isomorphism $(-) \cup U: H^*(W;\ZZ/l)\xrightarrow{\cong} H^{*+N}(W,\partial W;\ZZ/l)$. 
\end{proof}

\begin{definition}\label{definition-of-equivalence-of-spivak}
Two mod-$l$ Spivak normal spherical fibrations $(\gamma_1,\phi_1)$ and $(\gamma_2,\phi_2)$ over $Z$ are \textbf{equivalent} if there exists a stably orientation-preserving fiberwise homotopy equivalent $F:\gamma_1\rightarrow \gamma_2$ such that $F_*[\phi_1]=[\phi_2]$. 
\end{definition}

\begin{theorem}\label{uniqueness-of-l-adic-Spivak}
Any two mod-$l$ Spivak normal spherical fibrations $(\gamma_1,\phi_1),(\gamma_2,\phi_2)$ over $Z$ are equivalent.
\end{theorem}

The proof of this theorem is at the end of this subsection.

\begin{theorem}[$l$-adic Atiyah duality]\label{l-adic-Atiyah-duality}
Let $(\nu,\phi)$ be a mod-$l$ Spivak normal spherical fibration over $Z$. Let $\xi$ be an oriented $l$-adic spherical fibration over $Z$ and let $-\xi$ be an inverse of $\xi$. Then $M(\nu)$ is $l$-adic S-dual to $M(\nu \oplus (-\xi))$.
\end{theorem}

\begin{proof}
Note that $(\nu\oplus (-\xi))\oplus \xi=\nu\oplus \epsilon$, where $\epsilon$ is a trivial $l$-adic spherical fibration. Let $D(-)$ be the disc bundle of an $l$-adic spherical fibration, namely, it is defined as the mapping cylinder.  Consider the definition of the direct sum of spherical fibrations.
There is a fibration map $\rho:(D(\nu\oplus \epsilon),S(\nu\oplus \epsilon))\rightarrow (D(\nu\oplus (-\xi))\times D(\xi),S(\nu\oplus (-\xi))\times D(\xi)\bigcup D(\nu\oplus (-\xi))\times S(\xi))$ over the diagonal map $\Delta:Z\rightarrow Z\times Z$. Then we have the following commutative diagram.
\[
\begin{tikzcd}
    \widetilde{H}_{*+N}(M(\nu\oplus \epsilon)_l;\ZZ/l) \arrow[r,"\rho_*"] \arrow[d,"(-)\cap U_{\nu\oplus \epsilon}"] & \widetilde{H}_{*+N}((M(\nu\oplus (-\xi))\wedge M(\xi))_l;\ZZ/l) \arrow[d,"(-)\cap (U_{\nu\oplus(-\xi)}\times U_{\xi})"] \\
    H_*(Z_l;\ZZ/l) \arrow[r,"\Delta_*"] & H_*(Z_l\times Z_l;\ZZ/l)
\end{tikzcd}
\]

Consider the map $\nu:S^N\xrightarrow{\phi} M(\nu\oplus \epsilon)_l\xrightarrow{\rho} (M(\nu\oplus (-\xi))\wedge M(\xi))_l$. It suffices to prove that $\nu$ induces an isomorphism  $\nu_*([S^N]_{\ZZ/l})/(-):H^{*}(M(\nu \oplus (-\xi))_l;\ZZ/l)\rightarrow H_{N-*}(M(\xi)_l;\ZZ/l)$. Any cohomology class in $H^{*}(M(\nu \oplus (-\xi))_l;\ZZ/l)$ is uniquely represented by $x\cup U_{\nu \oplus (-\xi)}$ for some $x\in H^*(Z;\ZZ/l)$. Then 
\begin{align*}
    (\nu_*([S^N]_{\ZZ/l})/(x\cup U_{\nu\oplus (-\xi)}))\cap U_{\xi} & = \nu_*([S^N]_{\ZZ/l})/(x\cup U_{\nu\oplus (-\xi)}\cup U_{\xi}) \\
    & = (\nu_*([S^N]_{\ZZ/l})\cap (U_{\nu\oplus (-\xi)}\cup U_{\xi}))/x \\
    & = \Delta_*(\phi_*([S^N]_{\ZZ/l}\cap U_{\nu\oplus \epsilon})/x \\
    & = \Delta_*[Z]_{\ZZ/l}/x=[Z]_{\ZZ/l}\cap x
\end{align*}
By the mod-$l$ Poincar\'e dualities of $Z,Z_l$ and the Thom isomorphisms for all the spherical fibrations in this proof, we have proved that $\nu_*([S^N]_{\ZZ/l})/(-)$ is an isomorphism.
\end{proof}

In the case when $\xi$ is the trivial bundle of $Z$. Theorem \ref{l-adic-Atiyah-duality} specializes to:

\begin{corollary}
Let $(\nu,\phi)$ be a mod-$l$ Spivak normal spherical fibration over $Z$. Then $M(\nu)$ is $l$-adic dual to $Z_{+}$, where $Z_{+}$ is the disjoint union of $Z$ and a base point.
\end{corollary}

\begin{lemma}\label{uniqueness-of-l-adic-Spivak-Lemma-1}
If $(\xi_1,\phi_1),(\xi_2,\phi_2)$ are two mod-$l$ Spivak normal spherical fibrations of the same rank $k$ with $k$ much larger than the dimension $m$ of $Z$, then there is an orientation-preserving fiberwise homotopy equivalence $\xi_1\rightarrow \xi_2$.
\end{lemma}

\begin{proof}
Assume $\xi_1\oplus (-\xi_2)$ has rank $r$.
Theorem \ref{l-adic-Atiyah-duality} proves that there is an $l$-adic Spanier-Whitehead duality map $S^{m+k+r}\rightarrow (M(\xi_1)\wedge M(\xi_1\oplus (-\xi_2)))_l$. The map $\phi_1:\Sigma^{m+k}_l\rightarrow M(\xi_1)_l\wedge \Sigma^0_l$ has an $l$-adic $S$-dual map of spectra $D(\phi_1)':\Sigma^{\infty} M(\xi_1\oplus (-\xi_2))_l\rightarrow \Sigma^r_l$. Then there is some finite $r'$ and a map $D(\phi_1):M(\xi_1\oplus (-\xi_2)\oplus \epsilon^{r'}))_l\rightarrow (S^{r+r'})_l$ represents $D(\phi_1)'$. 

Let $[S^n]^{\vee}_{\ZZ/l}\in H^n(S^n;\ZZ/l)$ be the dual of the fundamental class $[S^n]_{\ZZ/l}$. Let $j:(S^{r+r'})_l\rightarrow M(\xi_1\oplus (-\xi_2)\oplus \epsilon^{r'})_l$ be an inclusion onto any fiber. Since $(\phi_1)_*([S^{m+k}]_{\ZZ/l})\cap U_{\gamma_1}=[Z]_{\ZZ/l}$, by the computations in the proof of Theorem \ref{l-adic-Atiyah-duality}, $(D(\phi_1))^*([S^{r+r'}]^{\vee}_{\ZZ/l}) =U_{\xi_1\oplus (-\xi_2)\oplus \epsilon^{r'}}$. Then $j^*\circ (D(\phi_1))^*([S^{r+r'}]^{\vee}_{\ZZ/l})$ is a generator of $H^n((S^{r+r'})_l;\ZZ/l)$. Hence, the composition $D((\xi_1\oplus (-\xi_2)\oplus \epsilon^{r'})_l)\rightarrow M(\xi_1\oplus (-\xi_2)\oplus \epsilon^{r'}))_l\rightarrow (S^{r+r'})_l$ is a fiber homotopy trivialization. One may apply a self homotopy equivalence on $(S^{r+r'})_l$ to make the previous fiber homotopy trivialization preserves the orientation on each fiber. This proves that $(\xi_1)_l$ and $(\xi_2)_l$ are stably oriented fiber homotopy equivalent.  This means that the classifying map $Z_l\xrightarrow{(\xi_1\oplus\epsilon)_l,(\xi_2\oplus \epsilon)_l}BSG(N')_l$ are homotopic for some large $N'$. By composing with the $l$-completion $f_l:Z\rightarrow Z_l$, $\xi_1$ and $\xi_2$ are stably oriented fiber homotopy equivalent. But the rank of $\xi_1,\xi_2$ are much larger than the dimension $Z$, this stable equivalence can descend to $\xi_1,\xi_2$.
\end{proof}

Next we prove that there exists an orientation-preserving fiberwise homotopy equivalence $b:\xi_1\rightarrow \xi_2$ such that $b\circ \phi_1=\phi_2\in \pi_{N+m}(M(\xi_2)_l)$.

Let $\epsilon^k$ be the trivial $l$-adic spherical fibration over $Z$. Let $\pi:M(\epsilon^k)\rightarrow (S^{k})_l$ be the fiber homotopy trivialization induced by the projection. Then $\pi^*([(S^{k})_l]^{\vee}_{\widehat{\ZZ}_l})$ is a Thom class of $\epsilon^k$, where $[(S^{k})_l]^{\vee}_{\widehat{\ZZ}_l}\in\widetilde{H}^k((S^{k})_l;\widehat{\ZZ}_l)$ is the dual to the fundamental class. Let $j:(S^{k})_l\rightarrow M(\epsilon^k)$ be the inclusion of the fiber. Let $\mathcal{A}_0(\gamma)$ be the submonoid of those orientation-preserving stable fiber homotopy equivalence of $\gamma$.

\begin{lemma}
Let $k$ be much larger than the dimension of $Z$. The map $\Psi_0:\mathcal{A}_0(\gamma)\rightarrow [M(\epsilon^k), (S^{k})_l]$ given by $\Psi_0(b)=\pi \circ M(b)$ is bijectively mapped onto the subset of those $g:M(\epsilon^k)\rightarrow (S^{k})_l$ such that $g_l\circ j$ is homotopy equivalent to $Id_{(S^{k})_l}$, where $g_l:M(\epsilon^k_l)\simeq M(\epsilon^k)_l\rightarrow (S^{k})_l$ is the $l$-adic completion of $g$ and $j:(S^{k})_l\rightarrow M(\epsilon^k_l)$ be the orientation-preserving inclusion of any fiber.
\end{lemma}

\begin{proof}
Injectivity: Let $b_0,b_1\in \Aut(\gamma)$ with $\pi \circ M(b_0)$ homotopic to $\pi \circ M(b_1)$. So they are homotopic after completion. Let $H:M(\epsilon^k)\times I\rightarrow M(\epsilon^k_l)\times I\xrightarrow{H_l} S^k$ be a homotopy between them. Notice that there is a canonical map $S(\epsilon^{k+1})\rightarrow M(\epsilon^k)_l\simeq M(\epsilon^k_l)$. Then $S(\epsilon^{k+1})\times I\rightarrow M(\epsilon^k_l)\times I\xrightarrow{H_l} (S^{k})_l$ produces a homotopy between $b_0\oplus Id_{\epsilon^1}$ and $b_1\oplus Id_{\epsilon^1}$.

Surjectivity: Given $g:M(\epsilon^k)\rightarrow (S^{k})_l$ with the assumption in the statement, $S(\epsilon^{k+1})\rightarrow M(\epsilon^k_l)\xrightarrow{g_l} (S^{k})_l$ has degree $1$ on each oriented fiber. This is the projection onto the fibers of a fiberwise homotopy equivalence $S(\epsilon^{k+1})\rightarrow S(\epsilon^{k+1})$, which is mapped to the stablization of $g$ by $\Psi$. But by the assumption of $k$, this fiberwise homotopy equivalence descends to  $S(\epsilon^{k})$.
\end{proof}

Let $\mathcal{A}_l(\gamma)$ be the \textbf{submonoid of those stable fiber homotopy equivalence of $\gamma$ which preserves $\ZZ/l$-orientations}. By the same argument, we can prove the following.

\begin{lemma}
Let $k$ be much larger than the dimension of $Z$. The map $\Psi_l:\mathcal{A}_l(\gamma)\rightarrow [M(\epsilon^k), (S^{k})_l]$ given by $\Psi_l(b)=\pi \circ M(b)$ is bijectively mapped onto the subset of those $g:M(\epsilon^k)\rightarrow (S^{k})_l$ such that $(g_l\circ j)^*=Id$ on $H^*((S^{k})_l;\ZZ/l)$.
\end{lemma}

\begin{lemma}\label{uniqueness-of-l-adic-Spivak-Lemma-2}
Fix a mod-$l$ Spivak normal spherical fibration $(\nu,\phi:S^{m+k}\rightarrow M(\nu)_l)$ over $Z$ with $k$ much larger than the dimension of $Z$, then the map $\Phi:\mathcal{A}_l(\nu)\rightarrow \pi_{m+k}(M(\nu)_l)$ given by $\Phi(b)=M(b)\circ \phi$ is a bijection onto the subset of those $h:S^{m+k}\rightarrow M(\nu)_l$ such that $h_*[S^{N+m}]\cap U_{\nu}=[Z]_{\ZZ/l}\in H_m(Z_l;\ZZ/l)\cong H_m(Z;\ZZ/l)$.
\end{lemma}

\begin{proof}
By Theorem \ref{l-adic-Atiyah-duality}, $M(\epsilon)$ is $l$-adic S-dual to $M(\nu)$. So there is commutative diagrams:
\[
\begin{tikzcd}
\mathcal{A}_l(\epsilon^{k'}) \arrow{r}{\cong} \arrow{d} & \mathcal{A}_l(\nu) \arrow{d} \\
\left [ \Sigma^{\infty}M(\epsilon^{k'})_l,\Sigma^{\infty}M(\epsilon^{k'})_l \right ] \arrow{r}{(-)^*} \arrow{d}{\Psi} & \left [ \Sigma^{\infty}M(\nu)_l,\Sigma^{\infty}M(\nu)_l \right ] \arrow{d}{\Phi}     \\
\left [ \Sigma^{\infty}M(\epsilon^{k'})_l,\Sigma^{k'}_l \right ] \arrow{r}{(-)^*} & \left [ \Sigma^{m+k}_l, \Sigma^{\infty}M(\nu)_l \right ]
\end{tikzcd}
\]
where the horizontal maps are the $S$-dual for maps with suitable suspensions. By the previous lemma, we know $\Psi$ is bijectively mapped onto the subset of those $g:M(\epsilon^k)\rightarrow (S^{k})_l$ such that $g_l\circ j$ preserves mod-$l$ fundamental class of $(S^{k})_l$. Then the $S$-duality between $g$ and $g^*$ is given by the following diagram which commutes up to homotopy.
\[
\begin{tikzcd}
S^{m+k+k'} \arrow{r}{\simeq} \arrow{d}{\phi} & S^{m+k}\wedge S^{k'} \arrow{dd}{(g^* \wedge Id_{S^{k'}})_l} \\
M(\nu)\wedge (S^{k'})_l\simeq  M(\nu\oplus \epsilon^{k'}) \arrow{d}{\Delta} \arrow{rd}{Id_{M(\nu_l)}\wedge f} & \\
M(\nu)\wedge M(\epsilon^{k'}) \arrow{r}{(Id_{M(\nu_l)}\wedge g)_l} & (M(\nu_l)\wedge S^{k'})_l\simeq M((\nu\oplus \epsilon^{k'})_l)
\end{tikzcd}
\]
Then $g_l\circ j$ preserves the mod-$l$ fundamental class of $(S^{k})_l$ if and only if $f$ preserves the mod-$l$ fundamental class of $(S^{k})_l$, if and only if $(\phi\circ (Id_{M(\nu_l)}\wedge f))_*[S^{m+k+k'}]_{\ZZ/l} \cap U_{M(\nu\oplus \epsilon^{k'})}=[Z]_{\ZZ/l}$ because of the mod-$l$ Thom isomorphism.
\end{proof}

\begin{proof}[Proof of Theorem \ref{uniqueness-of-l-adic-Spivak}]
By Lemma \ref{uniqueness-of-l-adic-Spivak-Lemma-1}, there exists a stable orientation-preserving fiberwise homotopy equivalence $b:\xi_1\rightarrow \xi_2$. By Lemma \ref{uniqueness-of-l-adic-Spivak-Lemma-2}, there exists a stable orientation-preserving fiberwise homotopy equivalence $b':\xi_2\rightarrow \xi_2$ such that $M(b')\circ M(b)\circ \phi_1$ is homotopy equivalent to $\phi_2$. Then $b'\circ b$ is an equivalence of mod-$l$ Spivak normal spherical fibrations $\xi_1,\xi_2$.
\end{proof}

\subsection{Mod-$l$ Poincar\'e Duality on Smooth Projective Varieties}\label{subsection-mod-l-poincare-duality-on-varieties}\;

The goal of this part is to prove the existence of mod-$l$ Poincar\'e duality on the $l$-adic \'etale homotopy type of a simply-connected smooth projective variety as in Theorem \ref{existence-of-mod-Poincare-duality}.

\begin{definition}
A spectrum $E$ is \textbf{$l$-finite} if its homotopy groups are all finite abelian $l$-groups.
\end{definition}

For a spectrum $E$, let $E[r,s]$ be the truncation of $E$ for any $-\infty\leq r\leq s\leq \infty$.

\begin{definition}[\cite{Joshua-Spanier-Whitehead}*{Definition 1.1}]
For a pro-space or pro-spectrum $U=\{U_i\}_{i\in I}$ and a spectrum $E$, define the $E$-homology/cohomology of $U$ by 
\[
h_*(U;E)=\varinjlim_{r\rightarrow -\infty}\varprojlim_i h_*(U_i;E[r,\infty]), h^*(U;E)=\varprojlim_{r\rightarrow \infty}\varinjlim_i h^*(U_i;E[-\infty,r])
\]
\end{definition}

The following result will be used in Definition \ref{definition-of-mod-l-S-duality-for-pro-spaces}.

\begin{fact}[\cite{Joshua-Spanier-Whitehead}*{Corollary 1.3}]\label{Joshua-lemma-for-l-finite-homology-theories}
Let $E$ be an $l$-finite spectrum and let $U=\{U_i\}_{i\in I}$ be a pro-space or a pro connective spectrum with a uniform lower bound. Then there are isomorphisms:
\[
h_*(U;E)\cong\varinjlim_{r\rightarrow -\infty} h_*(U;E[r,\infty]) \cong \varprojlim_{s\rightarrow \infty}\varinjlim_{r\rightarrow -\infty} h_*(U;E[r,s]) \cong \varinjlim_{r\rightarrow -\infty} \varprojlim_{s\rightarrow \infty} h_*(U;E[r,s]) \\
\]
\[
h^*(U;E)\cong\varprojlim_{s\rightarrow \infty} h^*(U;E[-\infty,s]) \cong \varprojlim_{s\rightarrow \infty}\varinjlim_{r\rightarrow -\infty} h^*(U;E[r,s]) \cong \varinjlim_{r\rightarrow -\infty} \varprojlim_{s\rightarrow \infty} h^*(U;E[r,s])
\]
\end{fact}

Let $U=\{U_i\}_{i\in I},U^*=\{U^*_j\}$ be pro finite-type spaces or pro finite-type connective spectra with a uniform lower bound. Let $E$ be an $l$-finite spectrum. 
A map of pro-spaces or pro-spectra $\rho:\Sigma^0\to (U^*\wedge U)^{\wedge}_l$ induces the following maps:
\[
D^{\rho}:h^q(U^*;E)\rightarrow h_{-q}(U;E),
\quad {}^{\rho}\!D:h^{q}(U;E)\rightarrow h_{-q}(U^*;E).
\]
We only review the construction of the first map here. The construction of the second one is similar. It suffices to restrict to the case when $q=0$. Let $\{f_{r,s}:U^*\rightarrow E[r,s]\}_{s}$ represent a class in $h^0(U^*;E)$ for some sufficiently small $r$. Then we have $\{f_{r,s}\wedge Id_{U}:U^*\wedge U\rightarrow E[r,s]\wedge U\}_s$. 
By induction on $s-r$ and the assumptions of $E$ and $U_j$, each homotopy group of $E[r,s]\wedge U^*_j$ is finite $l$-primary.
It follows that the morphism $f_{r,s}\wedge Id_{U}$ factors through $(f_{r,s}\wedge Id_{U})^{\wedge}_l: (U^*\wedge U)^{\wedge}_l \to E[r,s]\wedge U$. 
Therefore, we obtain $\{\Sigma^0\xrightarrow{\rho}(U^*\wedge U)^{\wedge}_l\xrightarrow{(f_{r,s}\wedge Id_{U})^{\wedge}_l}E[r,s]\wedge U \}_{s}$ for some sufficiently small $r$. By Fact \ref{Joshua-lemma-for-l-finite-homology-theories}, one gets an element in $h_{-q}(U;E)$, which defines $D^{\rho}$.

\begin{definition}[\cite{Joshua-Spanier-Whitehead}*{Definition 4.5}]\label{definition-of-mod-l-S-duality-for-pro-spaces}
A map $\rho:\Sigma^0\to (U^*\wedge U)^{\wedge}_l$ with $U,U^*$ pro finite-type spaces or pro finite-type connective spectra with a uniform lower bound is a \textbf{mod-$l$ Spanier-Whitehead duality (or mod-$l$ $S$-duality)} if $D^{\rho},{}^{\rho}D$ are isomorphisms for any $l$-finite spectrum $E$.
\end{definition}

\begin{definition}
Let $X$ be a smooth, connected scheme over a separably closed field.
Let $\gamma$ be an algebraic vector bundle over $X$. 
Define the \textbf{\'etale Thom space} $M(\gamma)_{\et}$ of $\gamma$ to be the pro-space given by the homotopy cofiber of the natural map $(\gamma-X)_{\et}\rightarrow \gamma_{\et}$.
\end{definition}

Recall the Euler short exact sequence (c.f. \cite[II, Theorem 8.13]{hartshorne2013algebraic})
\[0\to \mathcal{O}_{\PP^n_{\ZZ}}(-1)\to \mathcal{O}_{\PP^n_{\ZZ}}^{\oplus n+1}\to T_{\PP^n_{\ZZ}}(-1)\to 0.\]
Let $\alpha$ be the total space of $\Big(T_{\PP^n_{\ZZ}}(-1)\Big)^{\oplus n+1}$.
The analytification $\PP^n(\CC)^{an}$ is topologically embedded in $S^{2N}$ for some large $N$, unique up to isotopy. 
It is shown in \cite[p. 157]{Joshua-Spanier-Whitehead} that the analytification $(\alpha_{\CC})^{an}$ of the complex base change of $\alpha$ is 
, as real vector bundles, stably isomorphic to the normal bundle $\nu_{\PP^n/S^{2N}}$ of $\PP^n(\CC)^{an}$ in $S^{2N}$.
Let $M(\alpha_{\CC}^{an})$ be the Thom space of the topological bundle $\alpha_{\CC}^{an}$.
By Artin-Mazur's comparison theorem \cite{artin-mazur-etale-homotopy}*{Theorem 12.9, Corollary 12.13}, there are canonical weak equivalences of pro-spaces $M(\alpha_{\CC}^{an})^{\wedge}_l\simeq M(\alpha_{\CC})_{\et,l}^{\wedge}\simeq M(\alpha_{k})_{\et,l}^{\wedge}$. Then the usual Thom-Pontryagin map can be made as follows.:
\[
S^{2N}\rightarrow M(\nu_{\PP^n/S^{2N}})\simeq M(\alpha_{\CC}^{an})\rightarrow M(\alpha_{\CC}^{an})^{\wedge}_l\simeq M(\alpha_{\CC})_{\et,l}^{\wedge}\simeq M(\alpha_{k})_{\et,l}^{\wedge}
\]

Now assume that $X$ is projective and let us fix an embedding $X\hookrightarrow \PP^n_k$.
Let $\nu_X$ be the normal bundle of $X$ to the composite embeddings $X\hookrightarrow \PP^n_k\hookrightarrow \alpha$. 
By composing with the Thom-Pontryagin collapse map, we obtain (see \cite{Joshua-Spanier-Whitehead}*{p.~160})
\[
\rho: S^{2N}\rightarrow M(\alpha_{k})_{\et,l}^{\wedge}\rightarrow M(\nu_X)_{\et,l}^{\wedge}\xrightarrow{\Delta^{\wedge}_l} (M(\nu_X)_{\et}\wedge X_{\et,+})^{\wedge}_l
\]
where $\Delta: M(\nu_X)_{\et}\rightarrow M(\nu_X)_{\et}\wedge X_{\et,+}$ is the diagonal map.
A main result of \cite{Joshua-Spanier-Whitehead} is the following:

\begin{fact}[\cite{Joshua-Spanier-Whitehead}*{Theorem 4.7}]\label{construction-of-algebraic-normal-bundle}
$\rho:S^{2N}\rightarrow (M(\nu_X)_{\et}\wedge X_{\et,+})^{\wedge}_l$ is a mod-$l$ $S$-duality.
\end{fact}

In the remainder of this subsection, we show that $\holim M(\nu_X)_{\et}$ is $l$-adic S-dual to $\holim (X_{\et,+})^{\wedge}_l$.

\begin{lemma}
If $A$ and $B$ are simply-connected pro-spaces, then $(A\wedge B)^{\wedge}_l\rightarrow (A^{\wedge}_l\wedge B^{\wedge}_l)^{\wedge}_l$ is a weak equivalence.
\end{lemma}

\begin{proof}
Recall from \cite[\href{https://stacks.math.columbia.edu/tag/00DD}{Tag 00DD}]{stacks-project} that colimits commute with tensor product. Let $A=\{A_i\}$ and $B=\{B_j\}$. Hence, 
\begin{align*}
\widetilde{H}^{q}((A\wedge B)^{\wedge}_l;\ZZ/l) & \cong \widetilde{H}^{q}(A\wedge B;\ZZ/l)=\varinjlim_{i,j} \widetilde{H}^{q}(A_i\wedge B_j;\ZZ/l) \\
& \cong \varinjlim_{i,j} \bigoplus_{0\leq r\leq q}\widetilde{H}^{r}(A_i;\ZZ/l)\otimes \widetilde{H}^{q-r}(B_j;\ZZ/l) \\
& \cong \bigoplus_{0\leq r\leq q} \varinjlim_{i}\varinjlim_{j} \widetilde{H}^{r}(A_i;\ZZ/l)\otimes \widetilde{H}^{q-r}(B_j;\ZZ/l) \\
& \cong \bigoplus_{0\leq r\leq q} \varinjlim_{i} (\widetilde{H}^{r}(A_i;\ZZ/l)\otimes \varinjlim_{j} \widetilde{H}^{q-r}(B_j;\ZZ/l)) \\
& \cong \bigoplus_{0\leq r\leq q} (\varinjlim_{i} \widetilde{H}^{r}(A_i;\ZZ/l))\otimes (\varinjlim_{j} \widetilde{H}^{q-r}(B_j;\ZZ/l))\\
& \cong \widetilde{H}^*(A;\ZZ/l)\otimes \widetilde{H}^*(B;\ZZ/l)   
\end{align*}

By the same argument, $\widetilde{H}^*(A^{\wedge}_l\wedge B^{\wedge}_l;\ZZ/l)\cong \widetilde{H}^*(A;\ZZ/l)\otimes \widetilde{H}^*(B;\ZZ/l)$.
\end{proof}

\begin{corollary}
Let $A$ and $B$ be simply-connected pro-spaces. Assume that $H^q(A;\ZZ/l)$ and $H^q(B;\ZZ/l)$ are finite for any $q$. Then $(\holim A^{\wedge}_l \wedge \holim B^{\wedge}_l)_l\rightarrow \holim(A^{\wedge}_l \wedge B^{\wedge}_l)^{\wedge}_l$ is a homotopy equivalence.
\end{corollary}

\begin{proof}
By Lemma \ref{finite-type-theorem}, $\holim A^{\wedge}_l$, $\holim B^{\wedge}_l$ and $\holim(A^{\wedge}_l \wedge B^{\wedge}_l)^{\wedge}_l$ are all simply-connected, $l$-complete finite type. Then
\begin{align*}
\widetilde{H}^q(\holim(A^{\wedge}_l \wedge B^{\wedge}_l)^{\wedge}_l;\ZZ/l) & \cong \varinjlim \widetilde{H}^q ((A^{\wedge}_l \wedge B^{\wedge}_l)^{\wedge}_l;\ZZ/l) \,\,\,\text{(Lemma \ref{commutativity-of-limit-and-homology})} \\
& \cong \varinjlim \widetilde{H}^q (A^{\wedge}_l \wedge B^{\wedge}_l;\ZZ/l) \cong \bigoplus_{0\leq r\leq q} \varinjlim \widetilde{H}^{r} (A^{\wedge}_l;\ZZ/l)\otimes \varinjlim \widetilde{H}^{q-r} (B^{\wedge}_l;\ZZ/l) \\
& \cong \bigoplus_{0\leq r\leq q} \widetilde{H}^{r} (\holim A^{\wedge}_l;\ZZ/l) \otimes \widetilde{H}^{q-r} (\holim B^{\wedge}_l;\ZZ/l) \\
& \cong \widetilde{H}^{q} (\holim A^{\wedge}_l\wedge \holim B^{\wedge}_l;\ZZ/l)
\end{align*}
Then the conclusion holds.
\end{proof}

Assume that $\pi^{\et}_1(X)^{\wedge}_l=0$ and $X^{\wedge}_{\et,l}$ admits an $l$-local lifting. By Corollary \ref{finiteness-thereom-on-variety}, $X_{\et,l}^{\wedge}$ is $l$-adic weak equivalent to a simply-connected finite CW complex $Z$. The mod-$l$ $S$-duality $\rho:S^{2N}\rightarrow (M(\nu_X)_{\et}\wedge X_{\et,+})^{\wedge}_l$ induces a map
\begin{align*}
\rho':S^{2N}\rightarrow \holim(M(\nu_X)_{\et}\wedge X_{\et,+})^{\wedge}_l & \simeq (\holim (M(\nu_X)_{\et})^{\wedge}_l \wedge \holim (X_{\et,+})^{\wedge}_l)_l \\
& \simeq (\holim (M(\nu_X)_{\et})^{\wedge}_l \wedge (Z_{+})_l)_l    
\end{align*}

Consider the spectrum $E=H(\ZZ/l)$. Then we have the following commutative diagram.
\[
\begin{tikzcd}
H^{q}(\holim (M(\nu_X)_{\et})^{\wedge}_l ;\ZZ/l) \arrow{r}{D^{\rho'}} & H_{2N-q}(\holim (X_{\et,+})^{\wedge}_l;\ZZ/l) \arrow{d}{\cong} \\
H^{q}((M(\nu_X)_{\et})^{\wedge}_l ;\ZZ/l) \arrow{u}{\cong} \arrow{r}{D^{\rho},\cong} & H_{2N-q}((X_{\et,+})^{\wedge}_l;\ZZ/l) 
\end{tikzcd}
\]
By Theorem \ref{homological-discription-of-S-duality-2}, it follows that

\begin{corollary}\label{cor: duality for sullivan}
$\rho':S^{2N}\rightarrow (\holim (M(\nu_X)_{\et})^{\wedge}_l \wedge \holim (X_{\et,+})^{\wedge}_l)_l\simeq (\holim (M(\nu_X)_{\et})^{\wedge}_l \wedge Z_l)_l$ is an $l$-adic $S$-duality. In other words, $\holim (M(\nu_X)_{\et})^{\wedge}_l$ is $l$-adic S-dual to $Z$.
\end{corollary}

Recall the Thom-Gysin \'etale cohomology class $U_{\nu_X}^{alg}\in H^{2N-2m}_{X,\et}(\nu_X;\ZZ/l)$ which induces a Gysin-Thom isomorphism $(-)\cup U_{\nu_X}^{alg}:H^*_{\et}(X;\ZZ/l)\rightarrow H^{2N-2m+*}_{X,\et}(\nu_X;\ZZ/l)$. Consider the following commutative diagram: 
\[
\begin{tikzcd}
H^*_{\et}(X;\ZZ/l) \arrow[r,"(-)\cup U_{\nu_X}^{alg}"] \arrow[d,phantom,sloped,"\cong"] & H^{2N-2m+*}_{X,\et}(\nu_X;\ZZ/l) \arrow[d, phantom,sloped,"\cong"] \\
H^*(X_{\et};\ZZ/l) \arrow[r,"(-)\cup U_{\nu_X,\et}"] \arrow[d, phantom,sloped,"\cong"] & \widetilde{H}^{2N-2m+*}(M(\nu_X)_{\et};\ZZ/l) \arrow[d, phantom,sloped,"\cong"] \\
H^*(Z;\ZZ/l) \arrow[r,"(-)\cup U_{\nu_X}"] & \widetilde{H}^{2N-2m+*}(\holim(M(\nu_X)_{\et})^{\wedge,l};\ZZ/l)
\end{tikzcd}
\]
where $U_{\nu_X,\et}\in \widetilde{H}^{2N-2m}(M(\nu_X)_{\et};\ZZ/l)$ and $\widetilde{H}^{2N-2m}(\holim(M(\nu_X)_{\et})^{\wedge,l};\ZZ/l)$ are the Thom classes. It follows that all the horizontal arrows are isomorphisms. Let $[X_{\et}]_{\ZZ/l}=U_{\nu_X,\et}\backslash \rho_*[S^{2N}]$ and $[Z]_{\ZZ/l}=U_{\nu_X}\backslash (\rho')_*[S^{2N}]$.
Together with the mod-$l$ $S$-duality for pro-spectra and the $l$-adic $S$-duality for spectra, we have proven the followings and recovered a special case of \cite{Joshua-Spanier-Whitehead}*{Proposition 5.4} or \cite{Friedlander-etale-homotopy}*{Theorem 17.6}.

\begin{proposition}\label{existence-of-mod-Poincare-duality}
Let $X$ be a smooth, projective, connected scheme over a separably closed field $k$ of dimension $m$. Let $l$ be a prime number away from the characteristic of $k$. Assume that $\pi^{\et}_1(X)^{\wedge}_l=0$.
\begin{enumerate}
    \item There exists $[X_{\et}]_{\ZZ/l}\in \varprojlim H_{2m}(X_{\et};\ZZ/l)$ such that $(-)\cap [X_{\et}]_{\ZZ/l}:\varinjlim H^*(X_{\et};\ZZ/l)\rightarrow \varprojlim H_{2m-*}(X_{\et};\ZZ/l)$ is an isomorphism;
    \item Assume that $X^{\wedge}_{\et,l}$ admits an $l$-local lifting. Let $Z$ be a simply-connected finite CW which is $l$-adic weak equivalent to the \'etale homotopy type $X_{\et}$ of $X$. Then there exists $[Z]_{\ZZ/l}\in H_{2m}(Z;\ZZ/l)$ such that $(-)\cap [Z]_{\ZZ/l}:H^*(Z;\ZZ/l)\rightarrow H_{2m-*}(Z;\ZZ/l)$ is an isomorphism.
\end{enumerate}
\end{proposition}

\section{$l$-adic Manifold Structures on Varieties}

We develop the theory of $l$-adic formal manifold structures on an $l$-adic Poincar\'e duality space (see Definition \ref{definition-of-l-adic-formal-manifolds}) and prove the existence of an $l$-adic formal manifold structure on a smooth projective variety of characteristic away from $l$ (see Theorem \ref{existence-of-formal-manifold-structure}).

\subsection{$l$-adic Formal Manifolds}\;

In this subsection, we motivate and define the notion of $l$-adic formal manifold structures. The definition of $l$-adic formal manifolds is motivated by the following result in surgery theory.

\begin{theorem}
Let $Z$ be a simply-connected finite CW complex with Poincar\'e duality of dimension $m$ at least $5$. Then $Z$ is homotopy equivalent to a closed topological manifold if and only if the Spivak normal spherical fibration $\nu_Z$ lifts to a $TOP$ bundle.
\end{theorem}

\begin{proof}
The necessary part is obvious. We focus on the sufficient part. In the case of dimension $5$, we conclude by \cite[Theorem II.3.1]{Browder-Surgery-Simply-Connected}. From now on, we assume that the dimension is at least $6$. \cite[Theorem 2.2]{Wall-Poincare} shows that $Z$ is homotopy equivalent to an $m$-dimensional CW complex.
Remove an embedded $m$-dimensional disc in $Z$ 
and we get a Lefshetz duality pair $(Y,\partial Y)$, where $\partial Y$ is $S^{m-1}$. The bundle lifting condition induces a degree one normal map $(M,\partial M)\rightarrow (Y,\partial Y)$, where $M$ is a $TOP$ manifold with boundary. 
Applying the $\pi-\pi$ theorem (see \cite[Theorem 3.3]{Wall-Surgery}) or the trick from \cite[Theorem II.3.6]{Browder-Surgery-Simply-Connected} (i.e., $(M,\partial M)$ connect sum with some numbers of plumbings of disc tangential bundle of spheres to kill the surgery obstruction), $(M,\partial M)\rightarrow (Y,\partial Y)$ is normally bordant to a homotopy equivalence of pairs $(N,\partial N)\rightarrow (Y,\partial Y)$. By the genearalized Poincar\'e conjecture, $\partial N$ is homeomorphic to $S^{m-1}$. Coning off the boundaries we get a homotopy equivalence from a closed manifold to $Z$.
\end{proof}

The obstruction for lifting a spherical fibration to a $TOP$ bundle is understood at all primes.

\begin{proposition}[\cite{Brumfiel-Morgan}*{Theorem E and p.~9} for prime $2$; \cite{SullivanMITnotes}*{Theorem 6.5} for odd primes]
Let $\gamma$ be a spherical fibration over a CW complex $Z$. Then 
\begin{enumerate}
    \item there exist characteristic classes $k^G(\gamma)\in H^{4*+3}(Z;\ZZ/2)$ and $l^G(\gamma)\in H^{4*}(Z;\ZZ/8)$ such that $\gamma$ lifts to a $TOP$ bundle at prime $2$ if and only if $k^G(\gamma)$ vanishes and $l^G(\gamma)$ has a $\ZZ_{(2)}$-lifting;
    \item $\gamma$ lifts to a $TOP$ bundle at an odd prime $l$ if and only if $\gamma$ admits an $l$-local $KO$-theory orientation (see \cite{Madsen-Milgram-surgery}*{p.~82, the paragraph above (4.12)}).
\end{enumerate} 
\end{proposition}

Because of this, we propose the following definition of $l$-adic formal manifolds.

\begin{definition}\label{definition-of-l-adic-formal-manifolds}
Let $Z$ be a simply-connected finite CW complex with a mod-$l$ Poincar\'e duality of formal dimension $m\geq 5$.
\begin{enumerate}
    \item If $l$ is an odd prime number, then an \textbf{$l$-adic formal manifold structure} on $Z$ is an orientation $\Delta_Z\in KO_l^\wedge(\Sigma^{\infty} M(\nu_Z)_l)$ for its stable mod-$l$ Spivak normal spherical fibration $\nu_Z$ provided by Proposition \ref{Z/l-poincare-to-spivak}, where $KO_l^\wedge$ is the $l$-complete real $K$-theory.
    \item If $l=2$, then a \textbf{$2$-adic formal manifold structure} on $Z$ consists of an $\widehat{\ZZ}_2$-lifting $L_Z\in H^{4*}(Z;\widehat{\ZZ}_2)$ of the characteristic class $l^G(\nu_Z)\in H^{4*}(Z;\ZZ/8)$ and a homotopy class of a null homotopy of $k^G(\nu_Z):Z\rightarrow \prod_{k}(\ZZ/2;4k+3)$.
\end{enumerate}
\end{definition}

By definition, we have:

\begin{proposition}\label{prop: structure iff lift}
$Z$ has an $l$-adic formal manifold structure if and only if its mod-$l$ Spivak normal spherical fibration $\nu_Z:Z\rightarrow BSG_l$ has an $l$-lifting $Z\rightarrow BSTOP_l$.
\end{proposition}

\begin{example}
Any closed topological manifold $M$ with $\pi_1(M)^{\wedge}_l=0$ of dimension at least $5$ has an $l$-adic formal manifold structure for any prime $l$. Indeed, $M$ is $l$-adic weak equivalent to a simply-connected CW complex $Z$. The Poincar\'e duality of $M$ induces a mod-$l$ Poincar\'e duality on $Z$. By the uniqueness of mod-$l$ Spivak normal spherical fibrations on $Z$ (Theorem \ref{uniqueness-of-l-adic-Spivak}), the stable normal bundle of $M$ provides an $l$-lifting of the mod-$l$ Spivak normal spherical fibration $\nu_Z$ of $Z$.
\end{example}

We are curious about the following question:

\begin{question}
Let $Z$ be a $2$-adic Poincar\'e simply-connected finite CW complex. Then the mod-$2$ Spivak normal spherical fibration $\nu_Z$ has a mod-$8$ characteritic class $l^G(\nu_Z)\in H^{4*}(Z;\ZZ/8)$. If $Z$ has formal dimension $4k$, what is the meaning of the evaluation $\langle l^G(\nu_Z),[\ZZ]_{2} \rangle\in \ZZ/8$?
\end{question}

\begin{fact}[\cite{Brumfiel-Morgan}*{p.~59, 8.1}]
Let $Z$ be a finite CW complex with integral Poincar\'e duality and the Spivak normal spherical fibration $\nu_Z$, then $\langle l^G(\nu_Z),[\ZZ]_{2} \rangle$ is the signature of $Z$ modulo $8$.
\end{fact}

For any nonsingular symmetric bilinear form $(V,B(-,-))$ over $\ZZ/2$, there exists a characteristic element $v\in V$ such that $B(x,x)=B(x,v)$ for any $x\in V$.

\begin{definition}
Let $(W,B(-,-))$ be a nonsingular symmetric bilinear form over $\ZZ/8$. Let $w$ be an element of $W$ such that its mod-$2$ reduction $\overline{w}\in W\otimes \ZZ/2$ is a characteristic element. Define the \textbf{mod-$8$ signature of $(W,B)$} to be $B(w,w)\in \ZZ/8$.  
\end{definition}

The mod-$8$ signature of $(W,B)$ is well defined. Indeed, if $w'=w+2u$, then $B(w',w')=B(w,w)+4B(w,u)+4B(u,u)$ is $B(w,w)$ modulo $8$, since $B(w,u)=B(u,u)$ modulo $2$.

\begin{fact}[\cite{Milnor-Husemoller}*{p.~24, Lemma II.5.2}]
If $(U,B(-,-))$ is a nonsingular symmetric bilinear form over $\ZZ$, then the mod-$8$ signature of $U\otimes \ZZ/8$ is the signature of $(U,B)$ modulo $8$.
\end{fact}

By \cite{Brumfiel-Morgan}*{p.~59, 8.1}, the mod-$2$ reduction $l^G_{4n}\otimes \ZZ/2$ of $l^G_{4n}$ is $v^2_{2n}$, where $v_{2n}$ is the Wu class. This motivates us to conjecture that

\begin{conjecture}
Let $Z$ be a $2$-adic Poincar\'e simply-connected finite CW complex of dimension $4n$ with the mod-$2$ Spivak normal spherical fibration $\nu_Z:Z\rightarrow BSG_2$, then $\langle l^G(\nu_Z),[Z]_2 \rangle\in \ZZ/8$ is the mod-$8$ signature of the symmetric bilinear form on $H^{2n}(Z;\ZZ/8)$ induced by the Poincar\'e duality.
\end{conjecture}

If this conjecture is true, then we would have the following corollary: for a $2n$-dimensional smooth projective variety $X$ over a field of characteristic away from $2$, the evaluation of the $l^G$-class of the stable ``normal'' bundle $\nu_X$ (see the paragraph above Fact \ref{construction-of-algebraic-normal-bundle}) on the $\ZZ/8$-coefficient \'etale fundamental class of $X$ is the mod-$8$ signature of the $\ZZ/8$-coefficient \'etale cohomology of $X$.

\subsection{$l$-adic Manifold Structures on Smooth Projective Varieties}\;

Let $k$ be a separably closed field of characteristic $p>0$. Fix embeddings $W(\overline{\FF}_p)\rightarrow \CC$ and $\overline{\FF}_p\rightarrow k$. Let $l\neq p$ be a prime.

Let $G=GL(N)$ be the general linear group scheme over $\spec(\ZZ)$. Consider the algebraic stack $\left[ */G \right]$, there is a hypercovering $B_*G$ of $\left[ */G \right]$, where $B_*G$ is a simplicial scheme obtained by the Bar construction of the group scheme $G$. By \cite{Chough-Etale-Homotopy-Algebraic-Stacks}*{Theorem 2.3.40}, there is a canonical weak equivalence between the \'etale homotopy types $(B_*G_R)_{\et}$ and $\left[ */G_R \right]_{\et}$ for any ring $R$. 

By \cite{artin-mazur-etale-homotopy}*{Corollary 12.12} or \cite{Friedlander-etale-homotopy}*{Proposition 8.8}, the embedding $\overline{\FF}_p\rightarrow k$ induces a weak equivalence $(B_*G_k)^{\wedge}_{\et}\rightarrow (B_*G_{\overline{\FF}_p})^{\wedge}_{\et}$. By \cite{Friedlander-etale-homotopy}*{Proposition 8.8} the base-change maps $(B_*G_{\overline{\FF}_p})^{\wedge}_{\et,l} \rightarrow (B_*G_{W(\overline{\FF}_p)})^{\wedge}_{\et,l} \leftarrow (B_*G_{\CC})^{\wedge}_{\et,l}$ are $l$-adic weak equivalences. 
From \cite{artin-mazur-etale-homotopy}*{Corollary 12.12} or \cite{Friedlander-etale-homotopy}*{Theorem 8.4 or Proposition 8.8}, there is a canonical weak equivalence between the pro-spaces $(B_*G_{\CC})^{\wedge}_{\et}$ and $BGL(N,\CC)^{\wedge}\simeq BU(N)^{\wedge}$, where $BGL(N,\CC),BU(N)$ are the topological classifying spaces for complex vector bundles. To summarize, we have the following commutative diagram with each arrow a weak equivalence.
\begin{equation}\label{Bundle-Diagram}
\begin{tikzcd}
    (B_*G_k)^{\wedge}_{\et,l} \arrow[r] \arrow[d] & (B_*G_{\overline{\FF}_p})^{\wedge}_{\et,l} \arrow[r] \arrow[d] & (B_*G_{W(\overline{\FF}_p)})^{\wedge}_{\et,l} \arrow[d] & (B_*G_{\CC})^{\wedge}_{\et,l} \arrow[l] \arrow[d] \arrow[r] & BU(N)^{\wedge}_{l} \\
    \left[ */G_k\right]^{\wedge}_{\et,l} \arrow{r} & \left[ */G_{\overline{\FF}_p}\right]^{\wedge}_{\et,l} \arrow{r} & \left[ */G_{W(\overline{\FF}_p)}\right]^{\wedge}_{\et,l} & \left[ */G_{\CC}\right]^{\wedge}_{\et,l} \arrow{l} &
\end{tikzcd}    
\end{equation}

Assume that $\pi^{\et}_1(X)^{\wedge}_l=0$.
By Corollary \ref{finiteness-thereom-on-variety}, we know that the \'etale homotopy type $X_{\et}$ is $l$-adic weak equivalent to a simply-connected finite CW complex $Z$. By Theorem \ref{existence-of-mod-Poincare-duality}, $Z$ carries a mod-$l$ Poincar\'e duality of formal dimension $2m$, which is induced by the \'etale Poincar\'e duality on $X$. By Proposition \ref{Z/l-poincare-to-spivak} and Theorem \ref{uniqueness-of-l-adic-Spivak}, $Z$ has a mod-$l$ Spivak normal spherical fibration $\nu_Z$, unique up to equivalence.

Recall that \cite{Joshua-Spanier-Whitehead}*{\S4.5, Theorem 4.7} (alternatively, see the Fact \ref{construction-of-algebraic-normal-bundle} and the paragraphs above it) constructs a stable algebraic ``normal'' bundle $\nu_X:X\rightarrow \left[ */G_k \right]$ of $X$ and the \'etale Thom space $M(\nu_X)_{\et}$ of $\nu_X$ is a mod-$l$ S-dual of $X_{\et}$. 
By Diagram (\ref{Bundle-Diagram}), $\nu_X$ induces an ``$l$-adic complex bundle'' $\tau:Z\rightarrow BU(N)_l$ over $Z$. 

\begin{lemma}
The $l$-completed Thom space $M(\tau)_l$ of $\tau$ is canonically homotopy equivalent to $\holim M(\nu_X)^{\wedge}_{\et,l}$, where $M(\nu_X)_{\et}$ is the \'etale Thom space of $X$.
\end{lemma}

\begin{proof}
The map $(\nu_X)^{\wedge}_{\et,l}:X^{\wedge}_{\et,l}\rightarrow \left[ */G_k \right]^{\wedge}_{\et,l}\simeq BU^{\wedge}_l\rightarrow BSG^{\wedge}_l$ induces $\holim (\nu_X)^{\wedge}_{\et,l}:\holim X^{\wedge}_{\et,l}\rightarrow BSG_l$. It suffices to prove that $M(\holim (\nu_X)^{\wedge}_{\et,l})_l$ is canonically homotopy equivalent to $\holim M(\nu_X)^{\wedge}_{\et,l}$. Consider the following commutative diagram of pro-spaces
\[
\begin{tikzcd}
    S(\holim (\nu_X)^{\wedge}_{\et,l}) \arrow[r] \arrow[d,"\phi"] & \holim X^{\wedge}_{\et,l} \arrow[d,"\psi"] \\
    (\nu_X-X)_{\et} \arrow[r] & X_{\et},
\end{tikzcd}
\]
where the upper left corner is the associated sphere bundle.
By Lemma \ref{commutativity-of-limit-and-homology}, $\psi$ is an $l$-adic weak equivalence. By the five-lemma on the long exact sequence of homotopy groups, $\phi$ is also an $l$-adic weak equivalence. Taking homotopy cofibers, we get a pro-map $\xi:M(\holim (\nu_X)^{\wedge}_{\et,l})\rightarrow M(\nu_X)^{\wedge}_{\et,l}$. 
The five-lemma on the long exact sequence of the mod-$l$ homology groups for the two row maps in the diagram shows that $\xi$ is an $l$-adic weak equivalence. By Lemma \ref{commutativity-of-limit-and-homology} again, $\holim M(\nu_X)^{\wedge}_{\et,l}\rightarrow M(\nu_X)^{\wedge}_{\et,l}$ is also an $l$-adic weak equivalence.
\end{proof}

\begin{corollary}
The composition $\tau:Z\rightarrow BU(N)_l\rightarrow BSG(N)_l$ is a mod-$l$ Spivak normal spherical fibration of $Z$. 
\end{corollary}

\begin{proof}
By Corollary \ref{cor: duality for sullivan}, $M(\tau)$ is an $l$-adic $S$-dual of $Z$ induced by $\mu:S^{N+2m}\rightarrow (M(\tau)\wedge Z_{+})_l$. Let $\phi:S^{N+2m}\rightarrow M(\tau)_l$ be the composition of $\mu$ and the projection $(M(\tau)\wedge Z_{+})_l\rightarrow M(\tau)_l$. Then $\phi_*[S^{N+2m}]\cap U_{\tau}$, where $U_{\tau}$ is the Thom class, induces a mod-$l$ Poincar\'e duality on $Z$ by the $l$-adic $S$-duality and the Thom isomorphism. One can modify $\phi$ by composing with certain self map of $S^{N+2m}$ so that $\phi_*[S^{N+2m}]\cap U_{\tau}=[Z]_{\ZZ/l}$. 
\end{proof}

By the uniqueness of $l$-adic Spivak normal spherical fibrations on $Z$, $\tau$ induces a $BU_l$-lifting of $\nu_Z$, in particular, a $BTOP_l$-lifting. 

In the following theorem, we consider both positive and zero characteristic cases.
Let $k$ be a separably closed field of characteristic $p\geq 0$. Let $l\neq p$ be a prime number. We will use the following setup:

\hypertarget{field-embeddings}{(*)} If $p$ is positive, then fix embeddings $\overline{\FF}_p\rightarrow k$ and $W(\overline{\FF}_p)\rightarrow \CC$; if $p=0$, then fix an embedding $k\rightarrow \CC$ or $\CC\rightarrow k$ depending on the cardinality of $k$.

\begin{theorem}\label{existence-of-formal-manifold-structure}
Let $X$ be a connected, smooth, projective $m$-dimensional variety over $k$. Assume $m\geq 3$, $\pi^{\et}_1(X)^{\wedge}_l=0$ and that $X^{\wedge}_{\et,l}$ admits an $l$-local lifting. Let $Z$ be a simply-connected finite CW complex together with an $l$-adic weak equivalence $Z\rightarrow X^{\wedge}_{\et,l}$ and an induced mod-$l$ Poincar\'e duality.  
\begin{enumerate}
    \item There exists a canonical $l$-adic formal manifold structure on $Z$ which depends on the map $Z\rightarrow X^{\wedge}_{\et,l}$ and the condition \hyperlink{field-embeddings}{(*)}.
    \item If $p=0$, given the embedding $k\rightarrow \CC$, then the $l$-adic formal manifold structure on $Z$ is the same as the one induced by the underlying manifold structure on $X_{\CC}^{an}$ and the map $Z\rightarrow X^{\wedge}_{\et,l}\simeq (X_{\CC})^{\wedge}_{\et,l}\simeq (X_{\CC}^{an})^{\wedge}_l$.
    \item Assume $p>0$. Let $R$ be a mixed characteristic discrete valuation ring with residue field $k$. Let $W(\overline{\FF}_p)\rightarrow R\rightarrow \CC$ be embeddings compatible with the condition \hyperlink{field-embeddings}{(*)}.
    If $X$ has a lifting $\widetilde{X}_{R}$ over $R$, then the $l$-adic formal manifold structure on $Z$ is the same as the one induced by the underlying manifold structure on $(\widetilde{X}_{\CC}^{an})_l$ and the map $Z\rightarrow X^{\wedge}_{\et,l}\simeq (\widetilde{X}_{R})^{\wedge}_{\et,l}\simeq(\widetilde{X}_{\CC})^{\wedge}_{\et,l}\simeq (\widetilde{X}_{\CC}^{an})^{\wedge}_l$.
\end{enumerate}
\end{theorem}

\begin{proof}
(1) and (2). The case for $p>0$ follows from Proposition \ref{prop: structure iff lift}, Fact \ref{construction-of-algebraic-normal-bundle}, and the fact that the Spivak fibration in this case comes from a vector bundle $\nu_X$.
The case $p=0$ is similar. 
Namely, the map $Z\rightarrow \holim X^{\wedge}_{\et,l}\simeq \holim(X_{\CC})^{\wedge}_{\et,l}\simeq \holim (X_{\CC}^{an})^{\wedge}_l=(X_{\CC}^{an})_l\xrightarrow{\nu_{X^{an}_{\CC}}} BU_l$ induces an $l$-adic formal manifold structure on $Z$.

(3). By construction (Paragraphs above Fact \ref{construction-of-algebraic-normal-bundle}), the ``normal bundle'' $\nu_X$ over $X$ also has a lifting $\nu_{\widetilde{X}_R}$ over $R$. The analytification $(\nu_{\widetilde{X}_{\CC}})^{an}$ is the stable normal of $\widetilde{X}_{\CC}^{an}$. By the diagram \ref{Bundle-Diagram}, there is a commutative diagram:
\[
\begin{tikzcd}
X^{\wedge}_{\et,l} \arrow{d} \arrow{r} & (\widetilde{X}_{R})^{\wedge}_{\et,l} \arrow{d} &(\widetilde{X}_{\CC})^{\wedge}_{\et,l} \arrow{l} \arrow{d} \arrow{r} & (\widetilde{X}_{\CC}^{an})^{\wedge}_l \arrow{d} \\
\left [ */G_k\right ]^{\wedge}_{\et,l} \arrow{r} \arrow{d} & \left [ */G_R\right ]^{\wedge}_{\et,l} \arrow{d} & \left [ */G_{\CC}\right ]^{\wedge}_{\et,l} \arrow{l} \arrow{r} \arrow{ld} & BGL(N,\CC)_l\simeq BU(N)_l \\
\left [ */G_{\overline{\FF}_p}\right ]^{\wedge}_{\et,l} \arrow{r} & \left [ */G_{W(\overline{\FF}_p)}\right ]^{\wedge}_{\et,l} & & 
\end{tikzcd}
\]
This shows that the map $Z_l\simeq \holim (\widetilde{X}_{\CC}^{an})^{\wedge}_l=(\widetilde{X}_{\CC}^{an})_l\xrightarrow{(\nu_{\widetilde{X}})^{an}_{\CC}} BU_l$ is homotopic to $Z_l\simeq \holim X^{\wedge}_{\et,l}\xrightarrow{\nu_X} BU_l$.
\end{proof}

\begin{remark}
The $l$-adic formal manifold structure on a positive characteristic smooth, projective variety $X$ on  gives rise to an invariant, i.e., the $KO$-theory orientation $\Delta_X$, and/or another invariant, i.e. the $\widehat{\ZZ}_2$-coefficient characteristic class $L_X$.
\end{remark}

For a real vector bundle, the Hirzebruch $L$-genus $L$ is a combination of Pontryagin classes with coefficient in $\ZZ[\frac{1}{3},\frac{1}{5},...]=\ZZ_{(2)}$: $L=L(p_1,p_2,...)$. If this real vector bundle is the realization of a complex bundle, then the $L$-genus is a combination of Chern classes with coefficient in $\QQ$: 
\[
L=\widetilde{L}(c_1,c_2,...)
\]

\begin{proposition}
Let $X$ be a connected, smooth, projective $m$-dimensional variety over $k$. Assume that the characteristic of $k$ is not $2$. Then $L_X=\widetilde{L}(c_1,c_2,...)\in H^{4*}_{\et}(X;\QQ_2)$, where $c_1+c_2+...$ is the Chern class of the tangent bundle of $X$.  
\end{proposition}

\begin{proof}
This directly follows from the map $X^{\wedge}_{\et,2}\xrightarrow{TX} \left[ */G_k\right]^{\wedge}_{\et,l}\simeq BU^{\wedge}_l\rightarrow BSO^{\wedge}_l$.
\end{proof}

\begin{proposition}
Let $X$ be a connected, smooth, projective $m$-dimensional variety over $k$. Assume that $X$ admits a $l$-adic formal manifold structure $\Delta_X$ for $l\neq 2$. Let $l$ be an odd prime number away from the characteristic of $k$. Then $ph (\Delta_X)=\widetilde{L}(c_1,c_2,...)\in H^{4*}_{\et}(X;\QQ_l)$, where $ph$ is the Pontryagin character of real bundles and $c_1+c_2+...$ is the Chern class of the tangent bundle of $X$.  
\end{proposition}

\begin{proof}
This directly follows from the map $X^{\wedge}_{\et,2}\xrightarrow{TX} \left[ */G_k\right]^{\wedge}_{\et,l}\simeq BU^{\wedge}_l\rightarrow BSO^{\wedge}_l$ and the computation of the odd-prime local universal orientation for real vector bundles (see \cite{SullivanMITnotes}*{p.~203} or \cite{Madsen-Milgram-surgery}*{p.~84, 4.14}).
\end{proof}

\section{$l$-adic Formal Manifold Structures and Galois Symmetry}

In this section, we study various Galois symmetries on $l$-adic formal manifolds.

\subsection{$l$-adic Normal Structure Set}\;

Recall that the structure set $\mathbf{S}^{TOP}(M)$ of a simply-connected closed topological manifold $M$ consists of equivalence classes of homotopy equivalences $N\rightarrow M$ with $N$ a closed manifold, where $N\rightarrow M$ is equivalent to $N'\rightarrow M$  if there exists a homeomorphism $N\rightarrow N'$ such that the following diagram commutes up to homotopy.
\[
\begin{tikzcd}
    N \arrow[r] \arrow[d] & M \\
    N' \arrow[ru] &
\end{tikzcd}
\]

It is technically difficult to directly $l$-adic complete this structure set. Instead, recall the surgery exact sequence of based sets (\cite{Hauptvermutung-book}*{p.~82, Theorem 3}):
\[
0\rightarrow \mathbf{S}^{TOP}(M)\rightarrow [M,G/TOP]\rightarrow P_m,
\]
where $G/TOP$ is the fiber of $BTOP\rightarrow BG$ and $P_m$ is the simply-connected surgery obstruction group. Our idea is to take the $l$-adic completion of ``normal structure set'' $[M,G/TOP]$. 

The local homotopy type of the surgery space $G/TOP$ are known (see \cite{Kirby-Siebenmann}*{p.~329, 15.3}).

\begin{fact}\label{local-information-of-G/TOP}
\begin{enumerate}
    \item $(G/TOP)_{(2)}\simeq \prod_{k>0}(K(\ZZ_{(2)},4k)\times K(\ZZ/2,4k-2))$;
    \item If $l$ is odd, then $(G/TOP)_{(l)}\simeq BO_{(l)}$.
\end{enumerate}
\end{fact}

This motivates the following definitions.

\begin{definition}\label{definition-of-l-adic-homotopy-manifold-structures}
Let $Z$ be a simply-connected, mod-$l$ Poincar\'e finite CW complex of formal dimension $m\geq 5$, with an $l$-adic formal manifold structure.
\begin{enumerate}
\item if $l=2$, a \textbf{$2$-adic normal homotopy manifold structure} over $Z$ is a pair $(L,K)$ of positive graded classes $L\in H^{4*}(Z;\widehat{\ZZ}_2)$ and $K\in H^{4*-2}(Z;\ZZ/2)$.
    \item If $l$ is odd, an \textbf{$l$-adic normal homotopy manifold structure} on $Z$ is an element $\phi\in 1+\widetilde{KO_l^\wedge}(Z)$. The abelian group $1+\widetilde{KO_l^\wedge}(Z)$ is defined by the classifying space $BO^{\wedge}_l\times \{1\}\subset BO^{\wedge}_l\times \ZZ$. 
\end{enumerate}
\end{definition}

\begin{remark}
Note that the $H$-space structure on $BO^{\wedge}_l\times \{1\}$ is induced by tensor product of vector bundles.
\end{remark}

\begin{definition}\label{definition-of-l-adic-normal-structure-set}
The \textbf{$l$-adic normal structure set} $\mathbf{S}(Z)^\wedge_{N,l}$ of $Z$ is the set of all $l$-adic normal homotopy manifold structures over $Z$. 
\end{definition}

\begin{definition}
A map $g:Z_l\rightarrow Z'_l$ of $l$-completions of simply-connected, mod-$l$ Poincar\'e finite CW complexes is said to have \textbf{degree one} if $g_*\circ (f_l)_*[Z]_{\ZZ/l}= (f'_l)_*[Z']_{\ZZ/l}$, where $f_l:Z\rightarrow Z_l$ and $f'_l:Z'\rightarrow Z'_l$ are $l$-completions.   
\end{definition}

The following result implies that our definition of $l$-adic normal structure set contains the examples we want, i.e., an $l$-adic homotopy equivalence of $l$-adic formal manifolds is an $l$-adic normal homotopy manifold structure.

\begin{proposition}\label{example-of-homotopy-manifold-structure}
Let $Z$ and $Z'$ be simply-connected mod-$l$ Poincar\'e finite CW complexes of formal dimension $m\geq 5$ with $l$-adic formal manifold structures. If $g:Z'_l\rightarrow Z_l$ is a degree $1$ homotopy equivalence, then $g$ gives rise to an $l$-adic normal homotopy manifold structure on $Z$. 
More explicitly,
\begin{enumerate}
    \item if $l$ is odd, then $\phi\cdot \Delta_Z=(g^{-1})^* \Delta_{Z'}$, where $\phi$ is the $l$-adic normal homotopy manifold structure on $Z$ induced by $g$ and $\Delta_Z,\Delta_{Z'}$ correspond to the $l$-adic formal manifold structures of $Z,Z'$ respectively;
    \item if $l=2$, then $(1+8L)\cdot L_Z=(g^{-1})^* L_{Z'}$, where $L$ is part of the $2$-adic normal homotopy manifold structure on $Z$ induced by $g$ and $L_Z,L_{Z'}$ correspond to part of the $2$-adic formal manifold structures of $Z,Z'$ respectively.
\end{enumerate}
\end{proposition}

\begin{proof}
Let $\tau:Z_l\rightarrow BTOP_l$ and $\tau':Z'_l\rightarrow BTOP_l$ be the $l$-adic formal manifold structures on $Z,Z'$ respectively. Then $\tau$ and $\tau\circ g$ are two $l$-adic $TOP$-bundle liftings of the $l$-adic Spivak normal spherical fibrations $\nu_Z$ of $Z$. This corresponds to a map $Z_l\rightarrow (G/TOP)_l$. Then the rest follows from the surgery theory (see \cite{SullivanMITnotes}*{p.~218} and \cite{MOrgan-Sullivan-surgery}*{p.~543, the paragraph above Notes}) and Fact \ref{local-information-of-G/TOP}.
\end{proof}

\begin{corollary}\label{composition-of-manifolds}
Let $Z,Z',Z''$ be simply-connected, mod-$l$ Poincar\'e finite CW complexes of formal dimension $m\geq 5$ with $l$-adic formal manifold structures. Let $Z''_l\xrightarrow{g} Z'_l\xrightarrow{f} Z_l$ be degree $1$ homotopy equivalences.
\begin{enumerate}
    \item If $l$ is odd, let $\phi_f,\phi_{gf}$ be the $l$-adic normal homotopy manifold structures on $Z$ induced by $f$ and $gf$ respectively and let $\phi_g$ be the $l$-adic normal homotopy manifold structure on $Z'$ induced by $g$. Then $\phi_{gf}=\phi_f\cdot (f^{-1})^*\phi_{g}$.
    \item If $l=2$, let $(L_f,K_f),(L_{gf},K_{gf})$ be the $2$-adic normal homotopy manifold structures on $Z$ induced by $f$ and $gf$ respectively and let $(L_g,K_g)$ be the $l$-adic normal homotopy manifold structure on $Z'$ induced by $g$. Then $L_{gf}=L_{f}+(f^{-1})^*L_{g}+8(L_{f}\cdot (f^{-1})^*L_{g})$ and $K_{gf}=K_{f}+(f^{-1})^*K_{g}$.
\end{enumerate}
\end{corollary}

\begin{proof}
This is a direct corollary of Proposition \ref{example-of-homotopy-manifold-structure} and the $H$-space structure of $G/TOP$ (see \cite{Madsen-Milgram-surgery}*{Corollary 4.31 and 4.37} and \cite{MOrgan-Sullivan-surgery}*{Theorem 8.8}).
\end{proof}

\subsection{Galois Symmetry}\;

Recall that for each prime $l$, there is an action of $\widehat{\ZZ}^{\times}_l$ on the $l$-complete complex $K$-theory and $l$-complete real $K$-theory induced by the Adams operations $\psi^{\sigma_l}$, where $\sigma_l\in \widehat{\ZZ}^{\times}_l$ (see \cite{SullivanMITnotes}*{p.~156}). 
The following fact follows from Sullivan's proof of the Adams conjecture in \cite{Sullivan-adams-conjecture} (or see \cite{Madsen-Milgram-surgery}*{p.~106, 5.13}):

\begin{fact}
The map $BU^{\wedge}_l\xrightarrow{\psi^{\sigma_l}-1} BU^{\wedge}_l$ canonically factors through $(G/U)^{\wedge}_l$. The same is true for $BO$ and $G/O$.     
\end{fact}

When $l=2$, there is a composition of maps $f_{\sigma_2}:BU^{\wedge}_2\rightarrow (G/U)^{\wedge}_2\rightarrow (G/TOP)^{\wedge}_2$. Recall that there are characteristic classes $K^q\in H^{4*+2}((G/TOP)^{\wedge}_2;\ZZ/2)$ for the Kervaire surgery obstructions (see \cite{Hauptvermutung-book}*{p.~88, Corollary 1}). Let $\widetilde{K}^{\sigma_2}=f^*_{\sigma_2}K^q$.

Since $\widetilde{K}^{\sigma_2}$ is a combination of Stiefel-Whitney classes, this combination defines a characteristic class $K^{\sigma_2}$ in $H^{4*+2}(BSG_2;\ZZ/2)$.

\begin{lemma}
Let $\sigma_2,\tau_2\in \widehat{\ZZ}^{\times}_2$, $K^{\sigma_2\tau_2}=K^{\tau_2}+K^{\sigma_2}$.
\end{lemma}

\begin{proof}
Since $\psi^{\sigma_2\tau_2}-1=\psi^{\sigma_2}(\psi^{\tau_2}-1)+(\psi^{\sigma_2}-1)$, $\widetilde{K}^{\sigma_2\tau_2}=(\psi^{\sigma_2})^*\widetilde{K}^{\tau_2}+\widetilde{K}^{\sigma_2}$. Furthermore, $(\psi^{\sigma_2})^*=\psi^{\sigma_2}_H$. So $(\psi^{\sigma_2})^*$ is the identity for the $\ZZ/2$-coefficient.
\end{proof}

\begin{definition}
Define $K^{\sigma_2}_Z=(\nu_Z)^* K^{\sigma_2}\in H^{4*+2}(Z;\ZZ/2)$, where $\nu_Z:Z\rightarrow BG_2$ is the stable $2$-adic Spivak normal spherical fibration of $Z$.    
\end{definition}

Sullivan's proof of Adams conjecture motivates the construction of an abelianized Galois action of $\widehat{\ZZ}^{\times}_l$ on an $l$-adic normal structure set $\mathbf{S}(Z)^{\wedge}_{N,l}$.

\begin{definition}\label{Def: ablianized Galois action}
Let $\sigma_l\in \widehat{\ZZ}^{\times}_l$. The \textbf{abelianized Galois action of $\widehat{\ZZ}^{\times}_l$} on $\mathbf{S}(Z)^{\wedge}_{N,l}$ is defined as follows: 
\begin{enumerate}
    \item if $l$ is odd, $\sigma_l\cdot \phi =\psi^{\sigma_l}\phi\cdot \frac{\psi^{\sigma_l}\Delta_Z}{\Delta_Z} $, where $\phi\in 1+\widetilde{KO_p^\wedge}(Z)$ is an $l$-adic normal homotopy manifold structure on $Z$;
    \item if $l=2$, $\sigma_2\cdot L= \frac{-1+L_Z^{-1}\cdot \psi^{\sigma_2}_HL_Z}{8}+ L_Z^{-1}\cdot \psi^{\sigma_2}_HL_Z\cdot \psi^{\sigma_2}_HL$ and $\sigma_2\cdot K=K+K^{\sigma_2}_Z$, where $(L,K)$ is an $l$-adic normal homotopy manifold structure on $Z$ and $\psi^{\sigma_2}_H$ is the cohomological Adams operation on $H^{2*}(Z;\widehat{\ZZ}_2)$.
\end{enumerate}
\end{definition}

\begin{remark}
Notice that the coefficient of $-1+L_Z^{-1}\cdot \psi^{\sigma_2}_HL_Z$ at degree $4n$ is $\sigma_2^{2n}-1$, which is divisible by $8$.
\end{remark}

The Galois group of the underlying field acts on the $l$-adic formal manifolds of varieties naturally.

Let $k$ be a separably closed field of characteristic $p$. Let $l\neq p$ be a prime number. Assume the condition of field embeddings \hyperlink{field-embeddings}{(*)}.

Let $X,X'$ be smooth, projective, connected, pointed $m$-dimensional varieties of dimension $m$ over $k$. Assume that $(\pi^{\et}_1(X))^{\wedge}_l=(\pi^{\et}_1(X'))^{\wedge}_l=0$ and that $X^{\wedge}_{\et,l},(X')_{\et,l}$ both admit $l$-localization liftings. Let $Z,Z'$ be simply-connected finite CW complex $l$-adic weak equivalent to $X,X'$ respectively. By Theorem \ref{existence-of-formal-manifold-structure}, there exist $l$-adic formal manifolds on $Z,Z'$ induced from $X,X'$ respectively. 

\begin{proposition}\label{algebraic-homotopy-manifold-structure}
Let $f:X'\rightarrow X$ be an algebraic morphism over some field automorphism of $k$. Assume that $f_{\et}:X'_{\et}\rightarrow X_{\et}$ is an $l$-adic weak equivalence. Then $f$ represents an element in the $l$-adic normal structure set $\mathbf{S}(Z)^{\wedge}_{N,l}$ of $Z$, which depends on the field embedding condition \hyperlink{field-embeddings}{(*)}.
\end{proposition}

\begin{proof}
$f$ induces homotopy equivalences $(Z')_l\simeq \holim (X')^{\wedge}_{\et,l}\rightarrow \holim X^{\wedge}_{\et,l} \simeq Z_l$. By Proposition \ref{example-of-homotopy-manifold-structure}, this corresponds to an $l$-adic normal homotopy manifold structure on $Z$.
\end{proof}

\begin{definition}\label{definition-of-k-algebraic-elements}
Such an element in $\mathbf{S}(Z)^{\wedge}_{N,l}$ induced by an $f$ in Proposition \ref{algebraic-homotopy-manifold-structure} is called \textbf{$k$-algebraic} with respect to the $l$-adic weak equivalence $Z\rightarrow X^{\wedge}_{\et,l}$.
\end{definition}

\begin{definition}
Let $\sigma\in \Gal(k)$. There is a \textbf{Galois action of $\Gal(k)$ on the $k$-algebraic elements} in $\mathbf{S}(Z)^{\wedge}_{N,l}$ defined by $\sigma(X'\xrightarrow{f} X)= (X')^{\sigma}\xrightarrow{\sigma} X'\xrightarrow{f} X$.
\end{definition}

Consider the following homomorphisms of Galois symmetries \hypertarget{homomorphism-of-Galois}{(**)}: 
\begin{enumerate}
    \item if the characteristic $p$ of $k$ is positive, define $\omega_k:\Gal(k)\xrightarrow{\mu_k} \Gal(\overline{\FF}_p)=\widehat{\ZZ}\xrightarrow{e_p}\widehat{\ZZ}^{\times}_l$, where $\mu_k$ is the natural restriction induced by the field embedding \hyperlink{field-embeddings}{(*)} and $e_p$ induced by $e_p(x)=p^x$ for $x\in \ZZ$;
    \item if $p=0$, define $\omega_k:\Gal(k)\xrightarrow{\mu} \widehat{\ZZ}^{\times}= \prod_q \widehat{\ZZ}^{\times}_q\xrightarrow{\pi} \widehat{\ZZ}^{\times}_l$, where $\mu_k$ is the restriction to the roots of unity and $\pi$ is the projection.
\end{enumerate}

\begin{lemma}\label{Galois-action-on-homotopy-manifolds}
Let $\sigma\in \Gal(k)$. Let $X^{\sigma}\xrightarrow{\sigma} X$ be the algebraic morphism induced by the field automorphism $\sigma$. Then
\begin{enumerate}
    \item if $l$ is odd, the homotopy $l$-adic manifold structure on $Z$ induced by $X^{\sigma}\xrightarrow{\sigma} X$ is $\phi_{\sigma}=\frac{\psi^{\omega_k(\sigma)}\Delta_Z}{\Delta_Z}$;
    \item if $l=2$, the homotopy $2$-adic manifold structure on $Z$ induced by $X^{\sigma}\xrightarrow{\sigma} X$ is $(L_{\sigma},K_{\sigma})=(\frac{L_Z^{-1}\cdot \psi^{\omega_k(\sigma)}_HL_Z-1}{8},K^{\omega_k(\sigma)}_Z)$.
\end{enumerate}
\end{lemma}

\begin{proof}
If $p>0$, by Sullivan's proof of the Adams conjecture in \cite{Sullivan-adams-conjecture}*{p.~69 the third paragraph}  and Quillen-Friedlander's proof in \cite{Quillen-Adams-Conjecture-1}*{p.~112 \S3}\cite{Friedlander-Adams-Conjecture}, the natural maps $[*/G_k]^{\wedge}_{\et,l}\simeq [*/G_{\overline{\FF}_p}]^{\wedge}_{\et,l}\simeq [*/G_{W(\overline{\FF}_p)}]^{\wedge}_{\et,l}\simeq [*/G_{\CC}]^{\wedge}_{\et,l}\rightarrow BU^{\wedge}_l$ are equivariant via $\Gal(k)\xrightarrow{\mu_k} \Gal(\overline{\FF}_p)\cong\widehat{\ZZ}\leftarrow \Gal(\CC)\xrightarrow{1/\omega_{\CC}} \widehat{\ZZ}^{\times}_l$, where the action of $\widehat{\ZZ}^{\times}_l$ on $BU^{\wedge}_l$ is induced by the Adams operations. Since $W(\overline{\FF}_p)$ is the extension of $W(\FF_p)$ by all the roots of unity except the $p$-primary roots of unity, the map $[*/G_k]^{\wedge}_{\et,l}\rightarrow BU^{\wedge}_l$ is equivariant exactly via $\omega_k$. Similarly, if $p=0$, the natural map $[*/G_k]^{\wedge}_{\et,l}\rightarrow BU^{\wedge}_l$ is equivariant via $\omega_k$ as well. 

Assume that $Z^{\sigma}$ is the simply-connected CW complex $l$-adic weak equivalent to $X^{\sigma}$ and carries an $l$-adic formal manifold structure induced from $X^{\sigma}$. The following commutative diagram shows that $(\sigma^{-1})^*\Delta_{Z^{\sigma}}=\psi^{\omega_k(\sigma)}\Delta_Z$ and $(\sigma^{-1})^*L_{Z^{\sigma}}=\psi^{\omega_k(\sigma)}L_Z$. 
\[
\begin{tikzcd}
    Z^{\sigma}_l \arrow{r} \arrow{d}{\sigma} & (X^{\sigma})^{\wedge}_{\et,l} \arrow{r}{\nu_{X^{\sigma}}} \arrow{d}{\sigma} & \left [*/G_k \right ]^{\wedge}_{\et,l} \arrow{d}{\sigma} \arrow{r} & BU^{\wedge}_l \arrow{d}{\psi^{\omega_k(\sigma)^{-1}}} \\
    Z_l \arrow{r} & X^{\wedge}_{\et,l} \arrow{r}{\nu_{X}}  & \left [*/G_k \right ]^{\wedge}_{\et,l} \arrow{r} & BU^{\wedge}_l,
\end{tikzcd}
\]
where $\sigma:Z^{\sigma}_l\rightarrow Z_l$ is the composition $Z^{\sigma}_l\simeq \holim (X^{\sigma})^{\wedge}_{\et,l}\rightarrow \holim X^{\wedge}_{\et,l}\simeq Z_l$.
This proves the formula for $\phi_{\sigma}$ and $L_{\sigma}$. $K_{\sigma}=K^{\omega_k(\sigma)}_Z$ directly follows from this diagram and the definition of $K^{\omega_k(\sigma)}_Z$.
\end{proof}

\begin{theorem}\label{Galois-action-on-l-adic-normal-structure-set}
Let $X$ be a smooth, projective, connected $m$-dimensional variety of dimension $m$ over a separably closed field $k$ with the field embedding condition \hyperlink{field-embeddings}{(*)}. Assume $m\geq 3$, $(\pi^{\et}_1(X))^{\wedge}_l=0$ and that $X^{\wedge}_{\et,l}$ admits an $l$-local lifting. Let $Z$ be a simply-connected finite CW complex $l$-adic weak equivalent to $X$ together with an $l$-adic formal manifold structure induced from $X$. Then the Galois action of $\Gal(k)$ on the $k$-algebraic elements in $\mathbf{S}(Z)^{\wedge}_{N,l}$  factors through the abelianized Galois action of $\widehat{\ZZ}^{\times}_l$ on $\mathbf{S}(Z)^{\wedge}_{N,l}$ via $\omega_k:\Gal(k)\rightarrow \widehat{\ZZ}^{\times}_l$ in \hyperlink{homomorphism-of-Galois}{(**)}.
\end{theorem}

\begin{proof}
Let $\phi$ or $(L,K)$ be an $k$-algebraic element in $\mathbf{S}(Z)^{\wedge}_{N,l}$ represented by an algebraic morphism $f:X'\rightarrow X$ over some field automorphism $\tau$ of $k$. Let $\sigma\in \Gal(k)$. Let $\phi^{\sigma}$ or $(L^{\sigma},K^{\sigma})$ be the $k$-algebraic element in $\mathbf{S}(Z)^{\wedge}_{N,l}$ represented by $(X')^{\sigma}\xrightarrow{\sigma} X'\xrightarrow{f} X$. It suffices to check that 
$\phi^{\sigma}=\omega_k(\sigma)\cdot \phi$, $L^{\sigma}=\omega_k(\sigma)\cdot L$ and $K^{\sigma}=\omega_k(\sigma)\cdot K$.

First consider the special case when $\tau=Id$. Then the formulas for $\phi^{\sigma}$ and $L^{\sigma}$ hold by Corollary \ref{composition-of-manifolds} and Lemma \ref{Galois-action-on-homotopy-manifolds}. For $K^{\sigma}$, consider the following commutative diagram.
\begin{equation}\label{key-diagram}
\begin{tikzcd}
    ((X')^{\sigma})^{\wedge}_{\et,l} \arrow{r}{f^{\sigma}} \arrow{d}{\sigma} & (X^{\sigma})^{\wedge}_{\et,l} \arrow{d}{\sigma} \\
    (X')^{\wedge}_{\et,l} \arrow{r}{f} & X^{\wedge}_{\et,l} 
\end{tikzcd}    
\end{equation}
Then $K^{\sigma}=K+(f^{-1})^*K^{\omega_k(\sigma)}_{X'}$ is actually $K^{\omega_k(\sigma)}_{X}+ (\sigma^{-1})^*K(f^{\sigma})$.
$K(f^{\sigma})$ is represented by the ``difference'' of the two liftings $(X^{\sigma})^{\wedge}_{\et,l}\xrightarrow{\nu_{X^{\sigma}}} [*/G_k]^{\wedge}_{\et,l}\rightarrow BU^{\wedge}_l\rightarrow BSTOP^{\wedge}_l$ and $(X^{\sigma})^{\wedge}_{\et,l}\xrightarrow{(f^{\sigma})^{-1}} ((X')^{\sigma})^{\wedge}_{\et,l}\xrightarrow{\nu_{(X')^{\sigma}}} [*/G_k]^{\wedge}_{\et,l}\rightarrow  BU^{\wedge}_l\rightarrow BSTOP^{\wedge}_l$ of $(X^{\sigma})^{\wedge}_{\et,l}\rightarrow BSG^{\wedge}_l$. More explicitly, it is represented by the following diagram
\[
\begin{tikzcd}
((X')^{\sigma})^{\wedge}_{\et,l} \arrow{r}{\nu_{(X')^{\sigma}}} & BU^{\wedge}_l \arrow{d} \\
& BSTOP^{\wedge}_l \arrow{d} \\
(X^{\sigma})^{\wedge}_{\et,l} \arrow{uu}{(f^{\sigma})^{-1}} \arrow{ruu}{\nu_{X^{\sigma}}} \arrow{r} & BSG^{\wedge}_l 
\end{tikzcd}
\]
Then $(\sigma^{-1})^*K(f^{\sigma})$ is represented by the `difference' of the two liftings of $X^{\wedge}_{\et,l}\rightarrow BSG^{\wedge}_l$, induced by the map $X\xrightarrow{\sigma^{-1}}X^{\sigma}$ and by the two above liftings $(X^{\sigma})^{\wedge}_{\et,l}\rightarrow BSG^{\wedge}_l$. The diagram \ref{key-diagram} and the previous diagram shows that these two liftings over $X^{\wedge}_{\et,l}\rightarrow BSG^{\wedge}_l$ can be alternatively represented by the following diagram
\[
\begin{tikzcd}
(X')^{\wedge}_{\et,l} \arrow{r}{\nu_{(X')}} & BU^{\wedge}_l \arrow{r}{\psi^{\omega_k(\sigma)}} & BU^{\wedge}_l \arrow{d} \\
& & BSTOP^{\wedge}_l \arrow{d} \\
X^{\wedge}_{\et,l} \arrow{uu}{(f)^{-1}} \arrow{ruu}{\nu_{X}} \arrow{rr} &  & BSG^{\wedge}_l 
\end{tikzcd}
\]
That is, $(\sigma^{-1})^*K(f^{\sigma})=\psi^{\omega_k(\sigma)}_H(K)=K$. So the special case is proven.

Now consider the general case. Then $f:X'\rightarrow X$ factors as $X'\xrightarrow{\tau} (X')^{\tau} \xrightarrow{f'} X$, where $f'$ is over the identity of $k$. Let $\phi'$ or $(L',K')$ be the $k$-algebraic element in $\mathbf{S}(Z)^{\wedge}_{N,l}$ represented by $f'$. By the result of the special case, $\phi=\omega_k(\tau)\cdot \phi'$, $L=\omega_k(\tau)\cdot L'$ and $K=\omega_k(\tau)\cdot K'$. Then $\phi^{\sigma}=\omega_k(\tau\sigma)\cdot \phi'=\omega_k(\sigma)\cdot\omega_k(\tau)\cdot\phi'=\omega_k(\sigma)\cdot \phi$. It is similar for $L^{\sigma}$ and $K^{\sigma}$.
\end{proof}

Assume that $R$ is a discrete valuation ring with residue field $k$ and that the characteristic of $k$ is $p>0$. Let $W(\overline{\FF}_p)\rightarrow R\rightarrow \CC$ be embeddings compatible with the condition \hyperlink{field-embeddings}{(*)}. Consider the following homomorphism of Galois groups
\hypertarget{homomorphism-of-Galois-III}{(***)}: \\
$\Gal(k)\rightarrow \Gal(\overline{\FF}_p)\cong \Gal(\QQ^{un}_p/\QQ_p)=\widehat{\ZZ} \rightarrow \Gal(\QQ[\{\zeta_n\}_{p\nmid n}]/\QQ)\cong \widehat{\ZZ}^{\times} \leftarrow \Gal(\CC)$\\
where $\QQ^{un}_p=W(\overline{\FF}_p)$ is the maximal unramified extension of $\QQ_p$ and $\QQ[\{\zeta_n\}_{p\nmid n}]$ is the extension of $\QQ$ by the roots of unity which are not $p$-primary. 

Then \hyperlink{homomorphism-of-Galois-III}{(***)} connects the two homomorphisms of Galois groups in \hyperlink{homomorphism-of-Galois}{(**)} for $k$ and $\CC$. Thus the following is a direct corollary of Theorem \ref{Galois-action-on-l-adic-normal-structure-set}.

\begin{corollary}
With the above assumptions for $R,k$, let $X$ be a smooth, projective, connected $m$-dimensional varieties of dimension $m$ over $R$. Assume $m\geq 3$, $(\pi^{\et}_1(X))^{\wedge}_l=0$ and that $X^{\wedge}_{\et,l}$ admits an $l$-local lifting. Let $Z$ be a simply-connected finite CW complex $l$-adic weak equivalent to $(X_k)_{\et}$ together with an $l$-adic formal manifold structure induced from $X_k$. Then the Galois action of $\Gal(k)$ on the $k$-algebraic elements in $\mathbf{S}(Z)^{\wedge}_{N,l}$ is equivariant with the Galois action of $\Gal(\CC)$ on the $\CC$-algebraic elements in $\mathbf{S}(Z)^{\wedge}_{N,l}$, via the homomorphisms in \hyperlink{homomorphism-of-Galois-III}{(***)}.   
\end{corollary}

\bibliographystyle{amsalpha}
\bibliography{ref}

\Addresses

\end{document}